\documentclass[12pt,amsfonts, epsfig]{article}
\usepackage{hyperref}
\usepackage{amsthm}
\usepackage{epsfig,subfigure,latexsym,dcolumn,amsmath,longtable,setspace}
\usepackage{amsmath}
\usepackage{amsfonts}
\usepackage{amssymb}
\usepackage{latexsym}
\usepackage{mathrsfs}
\usepackage{amsmath,amscd}
\usepackage{amsmath, amscd, amssymb}
\usepackage[frame,cmtip,arrow,matrix,line,graph,curve,color]{xy}
\usepackage{graphpap, color}
\usepackage[mathscr]{eucal}
\usepackage{mathrsfs}
\usepackage{tikz}
\usepackage{xcomment}
\usepackage{verbatim}
\usetikzlibrary{matrix,arrows}

\numberwithin{equation}{section}

\newcommand{\CC}{\mathbb{C}}

\newcommand{\RR}{\mathbb{R}}
\newcommand{\ZZ}{\mathbb{Z}}

\newtheorem{prop}{Proposition}[section]
\newtheorem{theo}[prop]{Theorem}
\newtheorem{lemm}[prop]{Lemma}
\newtheorem{coro}[prop]{Corollary}

\newtheorem{defi}[prop]{Definition}
\newtheorem{conj}[prop]{Conjecture}

\newcommand{\la}{\langle}
\newcommand{\ra}{\rangle}
\newcommand{\nin}{\setminus}

\newcommand{\pa}{\partial}
\newcommand{\vv}{\vert\vert}

\newcommand{\ep}{\epsilon}
\newcommand{\lam}{\lambda}
\newcommand{\vol}{\text{vol}}

\title{An Asymptotic Faber-Krahn Inequality for the Combinatorial Laplacian on $\ZZ^2$}
\author{Yakov Shlapentokh-Rothman}
\begin{document}

\maketitle
\section{Introduction}
    \subsection{Statement of Theorem and Outline of Proof}

    Spectral graph theory studies graphs by mimicking ideas and techniques from the spectral theory of the Laplacian and other elliptic differential operators. For any graph, a real symmetric ``combinatorial Laplacian'' matrix is defined, and one relates the eigenvalues to properties of the graphs. One source of problems concerns taking theorems in spectral geometry and examining the extent to which their analogous statements hold in spectral graph theory.

    The Faber-Krahn inequality is a natural candidate for this program, and much work has been done exploring similar statements for graphs. The reader who is unfamiliar with the Faber-Krahn inequality and/or the spectral theory of the Laplacian might wish to first read the next section where a brief summary is given.

    We will consider a discrete asymptotic Faber-Krahn inequality for the combinatorial Laplcian on subgraphs of $\ZZ^2$. Informally speaking, we will show that as the areas of subgraphs go to infinity, the subgraphs with minimum first Dirichlet eigenvalue became ``circular.'' Before making a precise statement, we need some definitions: Suppose that $G$ is a finite subgraph of $\ZZ^2$. We will always assume that these subgraphs are induced, i.e. if any edge in $\ZZ^2$ connects two vertices of $G$, then this edge is in $G$. We denote the number of vertices by $|G|$. Number the vertices of $G$ arbitrarily from $1$, $\cdots$, $|G|$. If $i$ and $j$ are connected by an edge, we write $i \sim j$. Then the adjacency matrix is the $|G|\times |G|$ matrix $A$ defined by
    \[A_{ii} \equiv 0\]
    and for $i \neq j$
    \begin{equation*}
    A_{ij} \equiv A_{ji} \equiv \left\{
    \begin{array}{rl}
    1 & \text{if } i \sim j\\
    0 & \text{if } i \not\sim j
    \end{array} \right.
    \end{equation*}
    The combinatorial Dirichlet Laplacian on $\ZZ^2$ is the $|G|\times |G|$ matrix given by
    \[L_D \equiv 4I - A\]
    where $I$ denotes the $|G|\times |G|$ identity matrix. Motivation for this definition is provided in later sections. This is a real symmetric matrix and hence has $|G|$ real eigenvalues. We will show later that these do not depend on the ordering of the vertices. The lowest eigenvalue is denoted $\lam_D(G)$. We define

    \begin{equation*}
    \lam_D^{(n)} \equiv \inf_{G\subset \ZZ^2 \text{ with }|G| = n}\lam_D(G)
    \end{equation*}

    Later, a simple argument will show that this infimum is achieved for each $n$. Any subgraph $G_n$ whit $n$ vertices and $\lam_D(G_n) = \lam_D^{(n)}$ will be referred to as a ``minimizing subgraph.'' The goal of this paper is characterize the ``shape'' of minimizing subgraphs as $n\to \infty$. To make precise statements about the shape of a subgraph, it is useful to associate the subgraph with a domain in $\RR^2$ by taking the interior of the union of closed unit squares centered at each vertex. Here is an example of a subgraph along with the squares which constitute its associated domain.
    
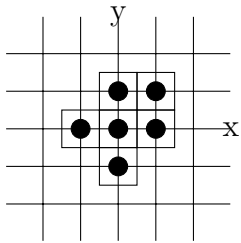
\begin{center}
    \begin{tikzpicture}
        [interior/.style={circle,draw=black,fill=black, inner sep=0pt,minimum size = 2.5mm},
        boundary/.style={circle,draw=black,fill=black!60, inner sep=0pt,minimum size = 2.5mm},
        exterior/.style={circle,draw=black,fill=white, inner sep=0pt,minimum size = 2.5mm},
        highlight/.style = {circle,draw= red, fill = red, inner sep=0pt, minimum size = 2.5mm}]
        \draw[step=.5cm] (-1.49,-1.49) grid (1.49,1.49);
        \draw (0,1.5) node {y};
        \draw (1.5,0) node {x};
        \draw (0,0) node [interior]{};
        \draw (0,0) node[inner xsep =.25cm,inner ysep = .25cm, draw] {};
        \draw (.5,0) node [interior]{};
        \draw (.5,0) node[inner xsep =.25cm,inner ysep = .25cm, draw] {};
        \draw (-.5,0) node [interior]{};
        \draw (-.5,0) node[inner xsep =.25cm,inner ysep = .25cm, draw] {};
        \draw (0,-.5) node [interior]{};
        \draw (0,-.5) node[inner xsep =.25cm,inner ysep = .25cm, draw] {};
        \draw (0,.5) node [interior]{};
        \draw (0,.5) node[inner xsep =.25cm,inner ysep = .25cm, draw] {};
        \draw (.5,.5) node [interior]{};
        \draw (.5,.5) node[inner xsep =.25cm,inner ysep = .25cm, draw] {};
    \end{tikzpicture}
\end{center}
    For a subgraph $G$, this associated domain is denoted by $\mathbf{G}$. Then we set
    \[\mathbf{G}^* \equiv \frac{1}{\sqrt{|G|}}\mathbf{G}\]
    This scales $\mathbf{G}$ so that it has area $1$. Now we are ready to state our main result:
    \begin{theo}\label{mainTheo}Let $\{G_n\}$ be any sequence of subgraphs in $\ZZ^2$ such that $|G_n| = n$ and $\lam_D(G_n) = \lam_D^{(n)}$. Let $D \subset \RR^2$ denote the unit disk. Then, after possibly translating the $G_n$, the measure of the symmetric difference of $\mathbf{G_n}^*$ and $D$ converges to $0$ as $n\to\infty$.
    \end{theo}
    Next we give an informal outline of the proof: For a measurable set $U$ in $\RR^n$, let $|U|$ denote the Lebesgue measure of $U$. For a bounded open set $\Omega \subset \RR^n$, let $\lam(\Omega)$ denote the first eigenvalue of the Laplacian as a differential operator with Dirichlet boundary conditions.

    The techniques used naturally divide the proof into four different parts. The first part is purely combinatorial and involves considerations of discrete versions of Steiner symmetrization. Steiner symmetrization takes an open domain $\Omega \subset \RR^n$ with smooth boundary and produces a new domain $\Omega^{\star}$ in the following fashion: For each $x = (x_1,\cdots,x_{n-1},0)$, let $l_x$ denote the line $y = x+te_n$, where $e_n$ denotes the $n$th standard basis vector. Then we partition $\Omega$ into ``slices,'' $S_x \equiv \Omega \cap l_x$.  To symmetrize, we replace each $S_x$ with an interval in $l_x$, symmetric about $\{x_n = 0\}$, and of size $|S_x|$. The symmetrization $\Omega^{\star}$ is the union of these centered intervals. This new domain is now symmetric about $\{x_n = 0\}$. Some key facts are $|\Omega^{\star}| = |\Omega|$ and $\lam(\Omega^{\star}) \leq \lam(\Omega)$. The second property opens the door for applications to eigenvalue minimization problems. Of course there is nothing special about the hyperplane $\{x_n = 0\}$. If we wish to symmetrize about another hyperplane $l$, we just change coordinates so that $l = \{x_n = 0\}$.

    We will consider two different types of discrete Steiner symmetrization. The first type will produce graphs that are ``almost symmetric'' with respect to either the $x$ or $y$ axis. The second type of symmetrization will produce graphs that are ``almost symmetric'' with respect to the lines $y=x$ or $y=-x$. Both procedures mimic regular Steiner symmetrization by partitioning the subgraph into ``slices'' and then modifying the slices to make them as symmetric as possible. We will show that neither form of symmetrization increases $\lam_D$. Furthermore, we will characterize certain subgraphs where symmetrization strictly lowers $\lam_D$. The upshot is that these subgraphs cannot be minimizing subgraphs.

    In the second part of the proof, we explore the geometry of minimizing subgraphs. The symmetrization results of the previous part successfully encapsulate much of the combinatorics. Hence, everything in the second section is a formal geometric consequences of facts from the first section. The most important result is that there exists $C > 0$ such that any minimizing graph on $n$ vertices is contained in a square with side length $C\sqrt{n}$. Equivalently, for any minimizing subgraph $G$, the diameter of $\mathbf{G}^*$ is less than $C'$ for some universal constant $C' > 0$. We will also establish that for any minimizing subgraph $G$, $\mathbf{G}$ is simply connected.

    The third part of the proof is analytical. For any domain $\Omega \subset \RR^2$ and $\ep > 0$, we define $B^{\ell_1}_{\ep}(\Omega)$ to be the interior of the set of all points with ${\ell}_1$\footnote{For $x \in \RR^n$, $\vv x\vv_{{\ell}_1} = \sum_{i=1}^n |x_i|$}distance less than $\ep$ to $\overline{\Omega}$. For subgraphs $G$ with $n$ vertices we will prove
    \begin{theo}\label{approxTheo}For some universal constant $C > 0$
    \[\frac{\lam\left(B^{\ell_1}_{2/\sqrt{n}}\left(\mathbf{G}^*\right)\right)}{n + C\lam\left(B^{\ell_1}_{2/\sqrt{n}}\left(\mathbf{G}^*\right)\right)} \leq \lam_D(G) \leq \frac{\lam\left(\mathbf{G}^*\right)}{n - C\lam\left(\mathbf{G}^*\right)}\]
    \end{theo}
    Fix some sequence of minimizing subgraphs $\{G_n\}$. In the final part of the proof we will establish the following two lemmas:
    \begin{lemm}The symmetric difference of $B^{\ell_1}_{2/\sqrt{n}}\left(\mathbf{G_n}^*\right)$ and $\mathbf{G_n}^*$ converges to $0$ as $n \to\infty$.
    \end{lemm}
    \begin{lemm}The sequence $\left\{\lam\left(B^{\ell_1}_{2/\sqrt{n}}\left(\mathbf{G_n}^*\right)\right)\right\}$ is uniformly bounded.
    \end{lemm}
    Assuming these lemmas we will use the following theorem of Melas from the end of \cite{n6} to piece everything together.
    \begin{theo}\label{faberKrahnStability}Let $\Omega \subset \RR^2$ be a bounded simply connected open domain and $B$ be a disk with the same area as $\Omega$. Suppose that $\lam(\Omega) \leq (1+\ep)\lam(B)$ for sufficiently small $\ep > 0$. Then there exists a disk $D_1 \subset \Omega$ such that
    \[|D_1| \geq (1-C\ep^{1/4})|\Omega|\]
    \end{theo}

    In appendix I, a sequence of subgraphs $\{D_n\}$ is constructed such that $|D_n| = n$ and $\lam(\mathbf{D_n}^*) \to \lam(D)$ as $n \to \infty$. Since each $G_n$ is a minimizing subgraph, we must have
    \[\lam_D(G_n) \leq \lam_D(D_n)\]
    From Theorem \ref{approxTheo} we get
    \[\frac{\lam\left(B^{\ell_1}_{2/\sqrt{n}}\left(\mathbf{G_n}^*\right)\right)}{n + C\lam\left(B^{\ell_1}_{2/\sqrt{n}}\left(\mathbf{G_n}^*\right)\right)} \leq \frac{\lam\left(\mathbf{D^*_n}\right)}{n - C\lam\left(\mathbf{D^*_n}\right)}\]
    Since the $\lam\left(B^{\ell_1}_{2/\sqrt{n}}\left(\mathbf{G_n}^*\right)\right)$'s are uniformly bounded, multiplying both sides by $n$ and taking $n \to \infty$ implies that
    \[\lim\sup_{n\to\infty}\lam\left(B^{\ell_1}_{2/\sqrt{n}}\left(\mathbf{G_n}^*\right)\right) \leq \lam(D)\]
    This allows for an application of  Melas' theorem. This produces disks $B_n \subset B^{\ell_1}_{2/\sqrt{n}}\left(\mathbf{G_n}^*\right)$ such that
    \[|B_n| \geq (1-C\ep(n)^{1/4})|B^{\ell_1}_{2/\sqrt{n}}\left(\mathbf{G_n}^*\right)|\]
    where $\ep(n) \to 0$ as $n \to\infty$. Since the symmetric difference of $\mathbf{G_n}^*$ and $B^{\ell_1}_{2/\sqrt{n}}\left(\mathbf{G_n}^*\right)$ converges to $0$ as $n \to \infty$, Theorem \ref{mainTheo} immediately follows.
    \subsection{Spectral Theory Background}

    Though this paper is mainly concerned with graph theory, the primary motivation comes from the spectral theory of the Laplacian. Hence, we will briefly summarize important results from the spectral theory of the Laplacian with Dirichlet boundary conditions. Proofs more than a few lines will generally be omitted, and references will be provided. No results in this section will be used directly, but will instead provide context for ideas introduced later.

    The Laplacian is the differential operator given by
    \[-\sum_{i=1}^n\frac{\pa^2}{\pa x_i^2}\]
    This is commonly denoted by $\Delta$. Suppose $\Omega \subset \RR^n$ is a bounded open domain. Let $C^{\infty}_c(\Omega)$ denote the vector space of infinitely differentiable complex valued functions with compact support in $\Omega$. Then the Laplacian maps $C_c^{\infty}(\Omega)$ to $C_c^{\infty}(\Omega)$. This restriction to functions of compact support is referred to as ``Dirichlet boundary conditions.'' Classically, a non-zero $u \in C_c^{\infty}(\Omega)$ is an eigenfunction for the Laplacian if
    \[\Delta u = \lambda u\text{ for }\lam \in \CC\]
    We refer to $\lambda$ as the eigenvalue of $u$. It turns out that the spectral theory of the Laplacian is much richer if we relax our notion of eigenfunctions.

    For any $u \in C_c^{\infty}(\Omega)$ integration by parts gives us
    \begin{align}
    (\Delta - \lambda)u &= 0 \Leftrightarrow\\
    \int_{\Omega}[(\Delta - \lam)u]v\ dx &= 0\ \forall v \in C_c^{\infty}(\Omega)\Leftrightarrow \\\label{weak}
    \int_{\Omega} \nabla u\cdot\nabla v\ dx - \int_{\Omega}\lam uv\ dx &= 0\ \forall v \in C_c^{\infty}(\Omega)
    \end{align}
    Note that the last expression makes sense for $u \in C^1(\Omega)$, the vector space of continuously differentiable complex valued functions. We say $u \in C^1(\Omega)$ is a ``weak solution'' of $\Delta-\lam$ if (\ref{weak}) holds. We can push this idea further. Suppose $f \in L^2(\Omega)$. If the reader is unfamiliar with $L^2$, define it to be the completion of $C_c^{\infty}(\Omega)$ with respect to the inner product
    \[\la u,v\ra = \int_{\Omega}u\overline{v}\ dx\]
    Unless $f$ lies in the image of the natural embedding $C^1(\Omega) \hookrightarrow L^2(\Omega)$, there is no a priori notion of $\frac{\pa f}{\pa x_i}$. Motivated by the calculation above we make the following definition
    \begin{defi}Let $f \in L^2(\Omega)$. We say that $\frac{\pa f}{\pa x_i} = u \in L^2(\Omega)$ if
    \[\int_{\Omega} uv\ dx = -\int_{\Omega} f\frac{\pa v}{\pa x_i}\ dx\ \forall v \in C_c^{\infty}(\Omega)\]
    \end{defi}

    In section 5.2 of \cite{n12} it is shown that if this ``weak derivative'' exists, it is uniquely defined up to a set of measure zero. If $f$ is in the image of $C_c^{\infty}(\Omega)$, then integration by parts implies that both the regular derivative and this weak derivative agree. Furthermore, many properties of derivatives hold for weak derivatives. See \cite{n12} for the specifics. Now we are ready to define the Sobolev space $H^1$.
    \begin{defi}
    \[H^1(\Omega) = \left\{f \in L^2(\Omega): \text{ such that }\frac{\pa f}{\pa x_i}\text{ exists weakly for } i = 1,\ \cdots,\ n\right\}\]
    \end{defi}
    This space comes equipped with an inner product
    \[\la u,v\ra_{H^1} = \int_{\Omega}u\overline{v}\ dx + \int_{\Omega}\nabla u\cdot\overline{\nabla v}\ dx\]
    In section 5.2 of \cite{n12} it is shown that this inner product makes $H^1$ into a Hilbert space, i.e. it is complete with respect to the norm
    \[\vv u\vv_{H^1} = \sqrt{\la u,v\ra}\]
    Note that $C_c^{\infty}(\Omega)$ is easily seen to lie in $H^1$. We define $H_0^1(\Omega)$ to be the closure of $C_c^{\infty}(\Omega)$ in $H^1(\Omega)$. We say that $u \in H_0^1(\Omega)$ is an eigenfunction of $\Delta$ with eigenvalue $\lambda$ if
    \[\int_{\Omega} \nabla u\cdot\nabla v\ dx - \int_{\Omega}\lam uv\ dx = 0\ \forall v \in C_c^{\infty}(\Omega)\]
    Thus we have managed to reformulate our eigenvalue problem over $H_0^1(\Omega)$ which is a Hilbert space. This allows for many techniques of real and functional analysis to be applied.

    We group some key results into one theorem.
    \begin{theo}\label{spec}The set of eigenfunctions form a countable set $\{\varphi_i\}_{i=1}^{\infty}$ with real monotonically increasing positive eigenvalues $\{\lam_i\}_{i=1}^{\infty}$ such that
    \begin{enumerate}
    \item The $\{\varphi_i\}$ form an orthonormal basis of both $L^2(\Omega)$ and $H_0^1(\Omega)$.
    \item $\lam_i \to \infty$ as $i \to \infty$.
    \item The $\lam_i$ obey the following ``minimax principle''
    \[\lam_k = \max_{S \in V_{k-1}}\ \min_{u \in S^{\perp}\text{ and } u \neq 0}\frac{\int_{\Omega} |\nabla u|^2\ dx}{\int_{\Omega} |u|^2\ dx},\]
    where $V_{k-1}$ denotes the set of $k-1$ dimensional subspaces of $H_0^1(\Omega)$.

    \item $\varphi_1$ is either strictly positive or strictly negative in the interior of $\Omega$.
    \end{enumerate}
    \end{theo}
    See chapter 6.5 of \cite{n12} for more background and proofs.

    The map $R: H_0^1(\Omega) \to \RR$ given by
    \[u \mapsto \frac{\int_{\Omega} |\nabla u|^2\ dx}{\int_{\Omega} |u|^2\ dx}\]
    is called the Rayleigh quotient. Let $\lam(\Omega)$ denote the first eigenvalue of the Laplacian on $\Omega$. Theorem \ref{spec} gives a variational formulation of $\lam(\Omega)$.
    \begin{theo}\label{variational}
    \[\lam(\Omega) = \inf_{u \in H_0^1(\Omega)\text{ and }u\neq 0} R(u)\]
    Furthermore, on page 356 of \cite{n12} it is shown that this infimum is uniquely achieved by constant multiples of $\varphi_1$.
    \end{theo}

    For $\Omega \subset \RR^2$ we have a physical interpretation of $\lam(\Omega)$. Namely, $\lam(\Omega)$ corresponds to the deepest bass note of a drum whose skin is in the shape of $\Omega$. Based on physical evidence, Rayleigh made the following conjecture for $n=2$.

    \begin{theo}\label{faberKrahn}(Faber-Krahn Inequality)
        Let $D \subset \RR^n$ be the ball of volume $1$ about the origin. Then
        \[\lam(D) = \min\{\lam(\Omega) : \Omega \subset \RR^n\text{ is a bounded open set of volume }1\}\]
    \end{theo}

    The key technique involved in the proof is radial decreasing rearrangements. For any bounded open set $O$ in $\RR^n$ we let $O^*$ denote the ball centered at the origin with the same volume. Let $\Omega$ be a bounded open set of volume $1$ and $u \in C_c^{\infty}(\Omega)$ with $u \geq 0$. We define
    \[\Omega(c) = \{x \in \Omega: u(x) \geq c\}\]
    and set
    \[u^{\star}(x) = \sup\{c \in \RR : x \in \Omega(c)^*\}\]
    This same procedure can be carried out for $u \in H_0^1(\Omega)$ but the definitions need to be slightly modified to reflect the fact that elements of $H_0^1(\Omega)$ are equivalence classes of functions.
    The two key properties that allow for a proof of the Theorem \ref{faberKrahn} are
    \begin{equation}\label{p1}
    \int_D |u^{\star}|^2\ dx = \int_{\Omega} |u|^2\ dx
    \end{equation}
and
    \begin{equation}\label{p2}
    \int_D |\nabla u^{\star}|^2\ dx \leq \int_{\Omega} |\nabla u|^2\ dx
    \end{equation}

    For a thorough discussion of rearrangements see \cite{n13}.

    Given these properties, the proof of Theorem \ref{faberKrahn} is easy. Choose any open domain $\Omega$ of volume $1$. Let $u$ be the eigenfunction corresponding to $\lam(\Omega)$. Then
    \begin{align*}
    \lam(D) &\leq \frac{\int_D |\nabla u^{\star}|^2\ dx}{\int_D |u^{\star}|^2\ dx}\text{ by Theorem \ref{variational}}\\
    &\leq \frac{\int_D |\nabla u|^2\ dx}{\int_D |u|^2\ dx}\text{ by (\ref{p1}) and (\ref{p2})}\\
    &= \lam(\Omega)
    \end{align*}

    Following this, a natural question is: To what extent is $D$ is the unique minimizer of $\lam(\Omega)$? Suppose $\Omega$ is a domain with $\lam(\Omega) = \lam(D)$. Let $u$ be the eigenfunction associated to $\lam(\Omega)$. From the proof of Theorem \ref{faberKrahn} we see that
    \[\int_{\Omega}|\nabla u|^2\ dx = \int_D |\nabla u^{\star}|^2\ dx\]

    This leads to the following question: For what $u \in H_0^1(\Omega)$ with $u \geq 0$ do we have
    \[\int_{\Omega}|\nabla u|^2\ dx = \int_{D} |\nabla u^{\star}|^2\ dx\]
    This question and associated generalizations have been studied extensively. See \cite{n14} for a recent paper addressing these questions. As a special case of the main theorem in \cite{n14} we have
    \begin{theo}Let $\Omega$ be a bounded open set of volume $1$. Suppose $u \in H_0^1(\Omega)$, $u \geq 0$, and
    \[\int_{\Omega}|\nabla u|^2\ dx = \int_{D} |\nabla u^{\star}|^2\ dx\]
    Then, after a translation, the symmetric difference of $\Omega$ and $D$ has measure zero. That is, for some $x_0$ the measure of
    \[[(\Omega + x_0) - D] \cup [D-(\Omega+x_0)]\]
    is zero.
    \end{theo}
    Thus we do in fact have
    \begin{theo}\label{strongFaberKrahn}Suppose $\Omega$ is an open domain in $\RR^n$ of volume $1$ so that
    \[\lam(\Omega) = \lam(D)\]
    Then, after a translation, the symmetric difference of $\Omega$ and $D$ has measure zero.
    \end{theo}

\section{The Combinatorial Laplacian}
Everything discussed in this section can found with many more details in \cite{n2}.
Let $G$ be a finite graph with no loops and at most one edge between any two vertices. $|G|$ denotes the number of vertices in $G$. If $x$ is connected to $y$ via an edge of $G$ we say $x \sim_{G} y$. The degree of a vertex is the number of neighboring vertices. This is denoted by $\deg_G(x)$. For both $\sim_{G}$ and $\deg_{G}$ we will drop the $G$ if it is clear from context.

Now we will give some fundamental definitions. After numbering the vertices of $G$ arbitrarily, we let $B$ be the $|G| \times |G|$ diagonal matrix where $B_{jj}$ is the degree of the $j$th vertex of $G$. The $|G| \times |G|$ adjacency matrix $A$ is defined by setting
\begin{equation*}
    A_{ij} \equiv A_{ji} \equiv \left\{
    \begin{array}{rl}
    1 & \text{if } i \sim j\\
    0 & \text{if } i \not\sim j
    \end{array} \right.
\end{equation*}
for $i\neq j$ and $A_{ii} \equiv 0$.

Then we define the Laplacian to be the $|G| \times |G|$ matrix $L = B - A$. We identify $\RR^n$ with real valued functions on the vertices of $G$ by sending the standard basis vector $e_i$ to the function whose value on vertex $i$ is $1$, and otherwise is $0$. Hence, without further comment we will treat functions on $G$ as vectors in $\RR^n$ and vice versa. A different numbering of the vertices amounts to permuting the basis vectors and thus does not change the conjugacy class of $L$. We will speak of ``the'' Laplacian associated to $G$ with the understanding that we are only concerned with the conjugacy class of $L$. Less abstractly
\[(Lf)(i) = \sum_{j\sim i}(f(i)-f(j))\]
There are many reasons why this deserves to be called the ``Laplacian'' of a graph. One is the following analogue of the mean value property.
\begin{prop}
Suppose
\[(Lf)(i) \geq 0\]
Then
\[f(i) \geq (1/\deg(i))\sum_{j\sim i}f(j)\]
\end{prop}
\begin{proof}This is immediate from the definition.
\end{proof}
Another reason is
\begin{lemm}\label{intByParts}(Integration by Parts)
\[\la Lf,f\ra = \sum_{i\sim j}\Big(f(i)-f(j)\Big)^2\]
\end{lemm}
\begin{proof} This follows from a direct calculation.
\begin{align*}
\la Lf,f\ra &= \sum_{i=1}^{|G|}(Lf)(i)f(i)\\
&= \sum_{i=1}^{|G|}\left[f(i)\sum_{j\sim i}\left(f(i) - f(j)\right)\right]\\
&= \sum_{i\sim j}f(i)(f(i)-f(j)) + f(j)(f(j)-f(i))\\
&= \sum_{i\sim j}\left(f(i)-f(j)\right)^2
\end{align*}
\end{proof}
The key theorem about symmetric matrices is the spectral theorem.
\begin{theo}(Spectral Theorem) Let $V$ be an $n$ dimensional real inner product space. Suppose $A$ is a symmetric $n$ by $n$ real matrix. Let $\la\cdot,\cdot\ra$ denote the inner product on $V$. We can find a basis of eigenvectors $u_1$, $u_2$, $\cdots$, $u_n$ with real eigenvalues $\lam_1 \leq \lam_2 \leq \cdots \leq \lam_n$ such that
  \begin{equation*}
    \la u_i,u_j\ra = \left\{
    \begin{array}{rl}
    1 & \text{if } i = j\\
    0 & \text{if } i \neq j
    \end{array} \right.
    \end{equation*}
Such a basis is called an ``orthonormal basis.''
\end{theo}
See page 114 of \cite{n5} for a proof.

We also have a ``minimax'' principle.
\begin{theo}\label{minimax}We keep the set up of the previous theorem. Then
\[\lam_l = \min\left\{\frac{\la Av,v\ra}{\la v,v\ra} : v \in \text{ span}(u_1,\ u_2,\ \cdots,\ u_{l-1})^{\perp}\text{ and }v \neq 0\right\}\]
Furthermore, any such $v$ achieving the minimum must be an eigenfunction associated to $\lam_l$.
\end{theo}
\begin{proof}Let $v \in V$. For every $j$
\begin{align*}
\left\la v - \sum_{i=1}^n\left\la v,u_i\right\ra u_i, u_j\right\ra &= \left\la v,u_j\right\ra - \left\la \sum_{i=1}^n \left\la v,u_i\right\ra u_i,u_j\right\ra\\
&= \left\la v,u_j\right\ra - \sum_{i=1}^n\left\la v,u_i\right\ra\left\la u_i,u_j\right\ra\\
&= \left\la v,u_j\right\ra - \left\la v,u_j\right\ra\\
&= 0
\end{align*}
Thus
\[v = \sum_{i=1}^n\left\la v,u_i\right\ra u_i\]
This implies
\begin{align*}
\left\la Av,v\right\ra &= \left\la A\left(\sum_{i=1}^n\left\la v,u_i\right\ra u_i\right),\sum_{j=1}^n\left\la v,u_j\right\ra u_j\right\ra\\
&=\left\la\sum_{i=1}^n\lam_i\left\la v,u_i\right\ra u_i,\sum_{j=1}^n\left\la v,u_j\right\ra u_j\right\ra\\
&=\sum_{i,j=1}^n\lam_i\left\la v,u_i\right\ra\left\la v,u_j\right\ra\left\la u_i,u_j\right\ra\\
&=\sum_{i=1}^n\lam_i\left\la v,u_i\right\ra^2\\
\end{align*}
Now suppose that we have $v \in V$ with $v \neq 0$ such that $i < l$ implies $\la v,u_i\ra = 0$. Then
\[v = \sum_{i=l}^n\left\la v,u_i\right\ra u_i\]
and
\begin{align*}
\left\la Av,v\right\ra &= \frac{\sum_{i=l}^n\lam_i\left\la v,u_i\right\ra^2}{\left\la v,v\right\ra}\\
                       &= \frac{\sum_{i=l}^n\lam_i\left\la v,u_i\right\ra^2}{\sum_{i=l}^n\left\la v,u_i\right\ra^2}\\
                       &\geq \lam_l\frac{\sum_{i=l}^n\left\la v,u_i\right\ra^2}{\sum_{i=l}^n\left\la v,u_i\right\ra^2}\\
                       &= \lam_l
\end{align*}
Note that the inequality in the second line is strict unless $\lam_k > \lam_l$ implies $\la v,u_k\ra = 0$. Hence we have equality if and only if $v$ is an eigenfunction of $\lam_l$.
\end{proof}

Let $\lam_i(G)$ denote the $i$th eigenvalue of the Laplacian matrix $L$ associated to a graph $G$. As a consequence of Theorem \ref{minimax} with $l=1$ and Lemma \ref{intByParts}
\[\lam_1(G) = \inf_{f\neq 0} \frac{\sum_{i\sim j}\Big(f(i)-f(j)\Big)^2}{\sum_i f(i)^2}\]
This immediately implies that all eigenvalues of $G$ are non-negative. By setting $f$ to be the constant function, we see that $\lam_1(G) = 0$ for all graphs $G$. Note that the analogous statement is false for the Laplacian on bounded open domains in $\RR^n$ with Dirichlet boundary conditions. One can then consider $\lam_2(G)$. Let $1$ denote the constant function on $G$. Theorem \ref{minimax} with $l=2$ and Lemma \ref{intByParts} gives
\[\lam_2(G) = \inf_{\la f,1\ra = 0}\frac{\sum_{i\sim j}\Big(f(i)-f(j)\Big)^2}{\sum_i f(i)^2}\]
$\lam_2(G)$ is called the ``algebraic connectivity'' of the graph $G$. For example, it is straightforward to show that $\lam_2(G) > 0$ if and only if $G$ is connected. See \cite{n2} for more along these lines. As one might guess, it turns out this version of the graph Laplacian is more naturally thought of as a discrete analogue for the Laplace-Beltrami operator on a closed manifold. Since we are interested in discretizing the Dirichlet Laplacian on bounded domains in $\RR^n$, we need to carefully think about the correct way to encode our boundary conditions.

\section{Preliminaries on the Combinatorial Dirichlet Laplacian}
In the continuous case, boundary conditions are critical for well posed eigenvalue questions. We must be sure that we have discretized the Dirichlet boundary conditions in the right fashion.

One approach assigns some subset of the vertices to be boundary vertices and then defines a real symmetric Dirichlet Laplacian matrix which acts on functions that vanish on the boundary vertices. A typical class of such discrete Faber-Krahn problems concerns finding a graph with boundary that minimizes the first Dirichlet eigenvalue among all graphs with $n$ interior vertices and $k$ boundary vertices. See the last chapter of \cite{n2} for a brief survey of results along these lines.

Another approach is to consider induced subgraphs of a larger graph, possibly infinite. Indeed, in the generalizations of the Faber-Krahn inequality to non-Euclidean spaces, one considers domains lying in some ambient Riemannian manifold \cite{n9}. We think of the larger graph as the ambient manifold. The main advantage of this formulation is that there is a natural, geometric way to define the boundary of a subgraph.

In passing we recall the definition of an induced subgraph. A subgraph $G \subset \Gamma$ is induced if
\[g_1 \sim_{\Gamma} g_2 \text{ and } g_1,\ g_2 \in G\Rightarrow g_1 \sim_G g_2\]
From this point on, all subgraphs will be assumed to have at least $2$ vertices, be finite, and be induced without further comment. Also, we will always assume that $\Gamma$ is regular.

We now define the boundary of a subgraph.
\begin{defi}\label{boundExtInt} If $G$ is a subgraph of $\Gamma$, then the boundary of $G$ is the set of all points in $\Gamma\nin G$ which are connected to $G$.
\[\pa G \equiv \{v \in \Gamma\nin G: v \sim_{\Gamma} g \text{ for some }g \in G\}\]
\end{defi}
We give an example to illustrate this. In the following diagram we take $\Gamma = \ZZ^2$. Let $G$ be the subgraph determined by the black points. Then the white points form $\pa G$.

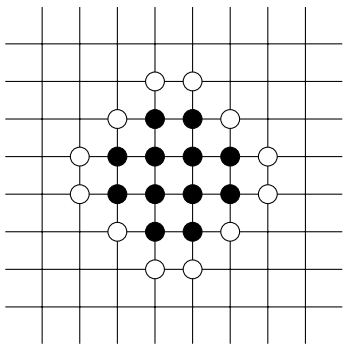
\begin{center}
\begin{tikzpicture}
[interior/.style={circle,draw=black,fill=black, inner sep=0pt,minimum size = 2.5mm},
 boundary/.style={circle,draw=black,fill=white, inner sep=0pt,minimum size = 2.5mm}]
\draw[step=.5cm] (-1.99,-1.99) grid (2.49,2.49);
\draw (0,0) node [interior]{};
\draw (0,.5) node [interior]{};
\draw (.5,0) node [interior]{};
\draw (.5,.5) node [interior]{};
\draw (0,1) node [interior]{};
\draw (.5,1) node [interior]{};
\draw (1,.5) node [interior]{};
\draw (1,0) node [interior]{};
\draw (.5,-.5) node [interior]{};
\draw (0,-.5) node [interior]{};
\draw (-.5,0) node [interior]{};
\draw (-.5,.5) node [interior]{};
\draw (0,1.5) node [boundary]{};
\draw (.5,1.5) node [boundary]{};
\draw (1,1) node [boundary]{};
\draw (1.5,.5) node [boundary]{};
\draw (1.5,0) node [boundary]{};
\draw (1,-.5) node [boundary]{};
\draw (.5,-1) node [boundary]{};
\draw (0,-1) node [boundary]{};
\draw (-.5,-.5) node [boundary]{};
\draw (-1,0) node [boundary]{};
\draw (-1,.5) node [boundary]{};
\draw (-.5,1) node [boundary]{};
\end{tikzpicture}
\end{center}

Now we give some more definitions. Suppose we have a subgraph $G$. Let $\overline{G}$ be the union of $G$ and its boundary. Then the Dirichlet Laplacian will be an operator which takes the set of functions on $\overline{G}$ which vanish on $\pa G$, to itself. It is defined by
\begin{equation*}
    (L_Df)(g) \equiv \left\{
    \begin{array}{rl}
    (Lf)(g) & \text{if } g \in G\\
    0 & \text{if } g \in \pa G
    \end{array} \right.
\end{equation*}
where $L$ is the regular Laplacian for $\overline{G}$. We define $L_D$ this way so that eigenfunctions of $L_D$ satisfy $L_Df = \lam f$ in the ``interior'' of $\overline{G}$ and vanish on $\pa G$ . Equivalently, we can avoid mentioning $\overline{G}$ by defining $L_D$ to act on functions defined on $G$ by
\[(L_Df)(x) \equiv \deg_{\Gamma}(x)f(x) - \sum_{y\sim_G\ x}f(y)\]
We emphasize the contrast with $L$ which is defined by
\[(Lf)(x) \equiv \deg_{G}(x)f(x) - \sum_{y\sim_G\ x}f(y)\]
Of course functions on $\overline{G}$ which vanish on $\pa G$ are trivially identified with functions on $G$. It is easily established that, under this identification, both of the above definitions are equivalent. If we number the vertices of $G$ we get a matrix for $L_D$. Let $B'$ be the diagonal matrix given by $\deg_{\Gamma}I$. Then the matrix for $L_D$ with respect to our chosen basis is $B' - A$, where $A$ is the adjacency matrix defined in the previous section. In the case of $\Gamma = \ZZ^2$, $L_D = 4I - A$.

This is a real symmetric matrix and thus the spectral theorem applies. For a subgraph $G$, we denote the smallest eigenvalue of $L_D$ by $\lam_D(G)$. The following probabilistic interpretation of $\lam_D(G)$ provides a useful intuitive crutch. Consider the following discrete Markov process: For our initial setup we place a ``particle'' at some vertex of $G$. Then on each iteration, the particles moves with equal probability to one of the $\Gamma$-neighbors of its current location. If a particle moves to a point in $\pa G$, then the particle ``falls off'' the graph and no longer occupies any vertices. Otherwise we keep iterating. Fix a vertex $i \in G$. Let $E_G^{(i)}$ be the random variable which gives the iteration of the random walk when the particle starting at $i$ falls off of $G$. We will prove
\begin{prop}\label{prob}Let $d$ be the degree of the vertices. For any two connected subgraphs $G$ and $H$ with $i \in G$, and $j \in H$, and $\lam_D(G)$, $\lam_D(H) < d$
\[\frac{\mathbb{P}\left(E_G^{(i)} \geq k\right)}{\mathbb{P}\left(E_H^{(j)}\geq k\right)} \sim \left(\frac{|\lam_D(G)-d|}{|\lam_D(H)-d|}\right)^{k}\]
\end{prop}
Later we will show that $\lam_D(G),\ \lam_D(H) \in (0,d]$ for all subgraphs. Hence we have
\begin{prop}\label{prob0}Assuming that $G$ and $H$ are connected, and $\lam_D(G) < \lam_D(H)$, we have
\[\frac{\mathbb{P}\left(E_G^{(i)} \geq k\right)}{\mathbb{P}\left(E_H^{(j)}\geq k\right)} \to \infty \text{ as }k\to\infty\]
\end{prop}

To understand why this might be an important graph invariant, let us consider a motivating example from computer science. Computer networks are often modeled by graphs where computers are represented by the vertices and edges represent connections between the computers where information exchange can occur. Now consider the following problem: There is a large network (graph) of computers and we have enough money to buy $n$ of the computers. We want information to travel quickly along our group of computers. However, we know nothing about computers that we do not own. Hence, there is always a  possibility of attack from computers in the boundary of our graph. Now we let the subgraph $G$ represent our network, the random walk represent the flow of information along $G$, and $E_G^{(i)}$ represent the event that information is stolen from an outside computer. Then, considering Propositions \ref{prob} and \ref{prob0}, it is natural to think of a lower $\lam_D$ as implying that the network is more ``securely connected.''

Before proving Propositions \ref{prob} and \ref{prob0}, we will need to recall some facts from the theory of matrices with non-negative entries. First we give the definition of a transitive matrix.
\begin{defi}Suppose $M$ is an $n\times n$ matrix with non-negative entries. We say $M$ is transitive if there exists some $N \geq 1$ such that $\sum_{k=1}^N M^k$ has all positive entries.
\end{defi}
For matrices whose diagonal entries are positive, we can formulate transitivity in graph theoretic terms.
\begin{prop}\label{transGraph}Let $M$ be a $n\times n$ matrix with non-negative entries. Let $H$ be a directed graph with $n$ vertices labeled $1$, $2$, $\cdots$, and $n$. Set $(i,j)$ to be an edge of $H$ if $M_{ji} > 0$. We claim that $M$ is transitive if $H$ is connected.
\end{prop}
\begin{proof}Let $e_i \in \RR^n$ be the standard basis vector with $1$ in the $i$th slot and $0$ everywhere else. By construction of $H$, $(i,j)$ is an edge if and only if $(Me_i)_j > 0$. In general, for any $v \in \RR^n$ with non-negative entries, $(Mv)_j > 0$ if and only if for some $i$ both $v_i > 0$ and $M_{ji} > 0$. Thus $(M^ke_i)_j > 0$ if and only if we can find $s_1$, $\cdots$, $s_{k-1}$ such that $M_{s_1i}$, $M_{s_2s_1}$, $\cdots$, $M_{s_{k-1}s_{k-2}}$, and $M_{js_{k-1}}$ are all positive. Equivalently, $(M^ke_i)_j > 0$ if and only if there exists some $s_1$, $\cdots$, $s_{k-1}$ such that $(i, s_1)$, $(s_1,s_2)$, $\cdots$, $(s_{k-2},s_{k-1})$, and $(s_{k-1},j)$ are all edges in $H$. Now suppose we can find $N$ such that $\sum_{k=1}^NM^k$ has all positive entries. Choose two vertices $i$ and $j$ in $H$. We can find some $k$ such that $(M^k)_{ji} > 0$, i.e. $(M^ke_i)_j > 0$. From the reasoning above this gives a path from $i$ to $j$. Now suppose that $H$ is connected. We want to find $N$ so that $\sum_{k=1}^NM^k$ has all positive entries. It is sufficient to produce an $l$ for every $(i,j)$ such that $M^l_{ji} > 0$. Since $H$ is connected we can find $s_1$, $\cdots$, $s_{k-1}$ such that $(i, s_1)$, $(s_1,s_2)$, $\cdots$, $(s_{k-2},s_{k-1})$, and $(s_{k-1},j)$ are all edges in $H$. Then by the reasoning above, this implies $M^k_{ji} > 0$ and we are done.
\end{proof}
\begin{defi}Let $G$ be a subgraph of $\Gamma$ and set $d = \deg_{\Gamma}$. We define $P_D = (-1/d)(L_D-dI)$. If we wish to emphasize the dependence on $G$ we write $P_D(G)$.
\end{defi}
\begin{coro}\label{pdtrans}Suppose $G$ is a connected subgraph. Then $P_D$ is transitive.
\end{coro}
\begin{proof} By construction, for $i \neq j$, $(P_D)_{ij} > 0$ if and only if $i \sim_G j$. Hence the graph $H$ of Proposition \ref{transGraph} is simply $G$. We immediately conclude that $P_D$ is transitive.
\end{proof}

The most important theorem about transitive matrices is the Perron-Frobenius Theorem.
\begin{theo}\label{perronFrobenius}(Perron-Frobenius) Let $A$ be a transitive non-negative matrix. From the fundamental theorem of algebra, $A$ has $n$ possibly complex eigenvalues $\lam_1$, $\cdots$, $\lam_n$. Let $\rho = \max_i |\lam_i|$. Then we claim that $\rho$ is an eigenvalue itself. Furthermore, $\rho$ has algebraic multiplicity $1$ and has a $1$ dimensional eigenspace. The eigenspace is spanned by a vector with all positive entries.
\end{theo}
See section 8.3 of \cite{n5} for a proof.
\begin{coro}\label{eig}If $G$ is a connected subgraph of $\Gamma$, then the eigenspace associated to $\lam_D(G)$ is spanned by a function $f$ whose value at every vertex is positive. Additionally, $\lam_D \in (0,d]$.
\end{coro}
\begin{proof}
Let the eigenvalues of $L_D$ be $\lam_D = \lam_1 \leq \lam_2 \leq \cdots \leq \lam_n$. Then the eigenvalues of $4P_D$ are $(d-\lam_n \leq d - \lam_{n-1}\leq \cdots\leq d - \lam_1 = d - \lam_D$. Furthermore, every eigenfunction for $G$ associated to $\lam_D$ is also an eigenfunction for $4P_D$ associated to $d - \lam_D$. Since $4P_D$ is transitive, positivity of $f$ follows from a direct application of the Perron-Frobenius Theorem. Also, we must have $d-\lam_D \geq 0$
\end{proof}
For a connected graph $G$, ``a principle eigenfunction for $G$'' will refer to any positive eigenfunction $f$ associated to $\lam_D$ with $\sum_x f^2(x) = 1$.

The following underlies Proposition \ref{prob}.
\begin{prop}\label{connectedProp2} For any two connected subgraphs $G$ and $H$
\[\frac{\vv P_D(G)^kv\vv_2}{\vv P_D(H)^kv\vv_2} \sim \left(\frac{|\lam_D(G) - d|}{|\lam_D(H) - d)|}\right)^{k}\]
\end{prop}
\begin{proof}Let $\lam_1 \leq \lam_2 \leq \cdots \leq \lam_n$ be the eigenvalues of $P_D(G)$ with the associated orthonormal basis of eigenfunctions $u_1$, $\cdots$, $u_n$. Similarly, let $\mu_1 \leq \mu_2 \leq \cdots \leq \mu_n$ be the eigenvalues of $P_D(H)$ with the associated orthonormal basis of eigenfunctions $w_1$, $\cdots$, $w_n$. Then
\[\vv P_D(G)^kv\vv^2_2 = \lam_1^{2k}\la v,u_1\ra^2 + \cdots + \lam_n^{2k}\la v,u_n\ra^2\]
and
\[\vv P_D(H)^kv\vv_2^2 = \mu_1^{2k}\la v,w_1\ra^2 + \cdots + \mu_n^{2k}\la v,w_n\ra^2\]
From Perron-Frobenius we know that $u_n$ and $w_n$ are both either strictly positive or strictly negative. Thus $\la v,u_n\ra^2$ and $\la v,w_n\ra^2$ are both strictly positive. We have $\lam_n > |\lam_i|$ for $i \neq n$ and $\mu_n > |\mu_i|$ for $i \neq n$. Since $\lam_n = (-1/4)(\lam_D(G) - d)$ and $\mu_n = (-1/4)(\lam_D(H) - d)$ the proposition follows.
\end{proof}

$P_D$ is related to the discrete Markov process described above. Let us start the process with a particle at vertex $i$. Then the arguments used in Proposition \ref{transGraph} go through almost unchanged. The probability that a particle which starts at $i$ has not fallen off of $G$ after $k$ iterations is $\vv P_D^ke_i\vv_1$\footnote{$\vv v\vv_1 = \sum_{i=1}^n |v_i|$}. In other words, $\mathbb{P}\left(P_D^{(i)} \geq k\right) = \vv P_D^ke_i\vv_1$. Recall that all norms on $\RR^n$ are equivalent\footnote{For any two norms $\vv\cdot\vv$ and $\vv\cdot\vv_*$ on $\RR^n$ there exists $C > 0$ such that $C^{-1}\vv v\vv \leq \vv v\vv_* \leq C\vv v\vv$ for all $v \in \RR^n$}. Hence, for any $i$ and $j$, Proposition \ref{connectedProp2} implies that
\[\frac{\mathbb{P}\left(E_G^{(i)} \geq k\right)}{\mathbb{P}\left(E_H^{(j)}\geq k\right)} \sim \left(\frac{|\lam_D(G) - d|}{|\lam_D(H) - d|}\right)^{k}\]
This establishes Proposition \ref{prob}.

The next corollary is useful when trying to explicitly calculate eigenfunctions on graphs with some nontrivial automorphisms. Recall that a bijection $\chi:G \to G$ is an automorphism if
\[\chi(x) \sim_G \chi(y) \Leftrightarrow x \sim_G y\]
\begin{coro}\label{eigenAuto}Let $f$ be a principle eigenfunction for a connected graph $G$. If $\chi: G \to G$ is an automorphism of $G$, then $f(g) = f(\chi(g))$ for all $g \in G$.
\end{coro}
\begin{proof} It follows directly from the relevant definitions that
\[L_D(f\circ \chi) = \lam_D(G)(f\circ \chi)\]
Theorem \ref{perronFrobenius} implies that $f\circ \chi$ is a multiple of $f$. Since the values of $f\circ \chi$ are a permutation of the values of $f$, we must have $f \circ \chi = f$.
\end{proof}
Before proceeding, we define the Rayleigh quotient of a non-zero function $f$, defined on $\overline{G}$ and vanishing on $\pa G$, by
\[R_G(f) = \frac{\sum\{(f(x)-f(y))^2: x\sim_{\overline{G}} y\}}{\sum_x f^2(x)}\]
To increase readability, we will drop the $\overline{G}$ when there is no ambiguity and write
\[R_G(f) = \frac{\sum_{x\sim y}(f(x)-f(y))^2}{\sum_x f^2(x)}\]
As with $\lam_1$, there is a variational characterization of $\lam_D$.
\begin{theo}\label{rayleighQuotientD}
\[\lam_D = \inf\Big\{R_G(f) : f \neq 0\text{ and }f|_{\pa G} = 0\Big\}\]
Also, $R_G$ achieves its minimum at $f$ if and only if $f$ is an eigenfunction.
\end{theo}
\begin{proof}The proof of Theorem \ref{minimax} goes through with slight changes.
\end{proof}

Now we present one more result which shows how closely related $\lam_D$ is to $\pa G$.
\begin{coro}\label{boundaryDer}Let $f$ be an eigenfunction of $G$. For $x \in G$ let $\pa(x)$ denote the number of vertices in $\pa G$ that are connected to $x$. Then
\[\sum_{x\in G}\pa(x)f(x) = \lam_D\vv f\vv_1\]
\end{coro}
\begin{proof}
We will compute the Euler-Lagrange equation for $R_G$. Consider a function $g$ on $\overline{G}$ vanishing on $\pa G$. Note that for small enough $t$, $f(x) + tg(x) > 0$. Thus, using Theorem \ref{rayleighQuotientD},

\begin{align*}
\frac{d}{dt}R_G(f+tg)|_{t=0} &= 0\Leftrightarrow \\
\frac{d}{dt}\frac{\sum_{x\sim y}\Big[f(x)-f(y) + t\Big(g(x)-g(y)\Big)\Big]^2}{\sum_x \Big(f(x)+tg(x)\Big)^2}\Big|_{t=0} &= 0 \Leftrightarrow \\
2\Big[\sum_x f^2(x)\Big]\Big[\sum_{x\sim y}\Big(f(x)-f(y)\Big)\Big(g(x)-g(y)\Big)\Big] &= \\
2\Big[\sum_{x\sim y}\Big(f(x)-f(y)\Big)^2\Big]\Big[\sum_xf(x)g(x)\Big] & \Leftrightarrow\\
\sum_{x\sim y}\Big(f(x)-f(y)\Big)\Big(g(x)-g(y)\Big) &=\\
\frac{\Big[\sum_x f(x)g(x)\Big]\Big[\sum_{x\sim y}\Big(f(x)-f(y)\Big)^2\Big]}{\sum_x f^2(x)}&= \lam_D\sum_x f(x)g(x)\Leftrightarrow \\
\sum_{x\sim y}\Big(f(x)-f(y)\Big)\Big(g(x)-g(y)\Big) &= \lam_D\sum_x f(x)g(x)
\end{align*}
The corollary follows by setting $g$ to be identically $1$ on $G$.
\end{proof}
Note that this does not have a well known analogue for the continuous Laplacian.
In passing, we mention that there are other definitions of $\lam_D$ in common usage. However, in the case of regular graphs all of these agree up to a constant. Since we will be concerned with subgraphs of $\ZZ^2$, the choice of definition is not important for us.
\section{Basics of the Faber-Krahn Problem on $\ZZ^2$}
 The graph Faber-Krahn problem now takes the following general form: For a fixed (usually infinite) graph $\Gamma$ define
 \[\lam_D^{(n)} = \inf\{\lam_D(G) : \text{ G is a subgraph and }|G| = n\}\]
 Then one wants to
 \begin{enumerate}
    \item Understand the asymptotics of the sequence $\left\{\lam_D^{(n)}\right\}_{n=1}^{\infty}$.
    \item Find all graphs $G$ with $|G| = n$ such that $\lam_D(G) = \lam_D^{(n)}$ or obtain information about such $G$ asymptotically.
 \end{enumerate}
We refer to such subgraphs as ``minimizing subgraphs.''

At present, this problem for general $\Gamma$ appears quite intractable. Furthermore, there is no reason to expect a nice answer for general $\Gamma$. Hence, there have been attempts to solve the Faber-Krahn problem on certain well understood infinite graphs, such as trees. Some of the results indicate that the situation is more complicated than one might naively expect. For example, minimizers in a tree are usually \emph{not} geodesic balls even though they are very close. See the final chapter of \cite{n2} for a survey and a collection of references.

For the rest of this paper we will take $\Gamma = \ZZ^2$. Since it is the graph analogue of Euclidean space, we can appeal to geometric intuition from $\RR^2$. The continuous Faber-Krahn inequality and the probabilistic interpretation of $\lam_D$ suggest that minimizing graphs should be ``circular'' in some sense, at least for large enough graphs. Indeed, we will show that for large $n$ the minimizing subgraphs on $n$ vertices must be ``close'' to an $\RR^2$ disk. ``Close'' is to be interpreted as in Theorem \ref{mainTheo}.

Before diving into the proof of Theorem \ref{mainTheo}, we will warm up with a couple easy propositions.
\begin{prop}For every $n$ there exists a subgraph $G$ with $|G| = n$ and $\lam_D(G) = \lam_D^{(n)}$.
\end{prop}
\begin{proof}Since isomorphic subgraphs are easily seen to have the same $\lam_D$, we need only consider subgraphs up to isomorphism. However, there are clearly only finitely many isomorphism classes of subgraphs in $\ZZ^2$ with $n$ vertices. Hence, $\lam_D^{(n)}$ must be achieved by some subgraph.
\end{proof}
\begin{prop}\label{eigenDecrease} $\left\{\lam_D^{(n)}\right\}_{n=1}^{\infty}$ is a strictly decreasing sequence.
\end{prop}
\begin{proof}Fix $n$ and choose a subgraph $G_n$ with $\lam_D(G_n) = \lam_D^{(n)}$. Let $l: G_n \to \ZZ^2$ be a map whose restriction to one of the components of $G_n$ is one of
\begin{enumerate}
\item $(x,y) \mapsto (x+1,y)$
\item $(x,y) \mapsto (x-1,y)$
\item $(x,y) \mapsto (x,y+1)$
\item $(x,y) \mapsto (x,y-1)$
\end{enumerate}
and whose restriction to the other components is the identity. Set $G_n' = l(G_n)$. We claim that $\lam_D(G_n') \leq \lam_D(G_n)$. To see this, let $f$ be a principle eigenfunction for $G_n$. Then define $f': G_n' \to \RR$ by $f' = f\circ l^{-1}$. It is easy to see that $x \sim_{G_n} y$ implies that $l(x) \sim_{G'_n} l(y)$. Furthermore, note that if $l(x) \sim_{G'_n} l(y)$ and $x \not\sim_{G_n} y$, then $x$ and $y$ must both lie in $\pa G_n$. These observations, along with Theorem \ref{rayleighQuotientD}, give the following string of inequalities
\begin{align*}
\lam_D(G'_n) &\leq R_{G'_n}(f')\\
 &= \sum\{(f'(x)-f'(y))^2: x\sim_{\overline{G'_n}} y\}\\
 &= \sum\{(f(l^{-1}(x)) - f(l^{-1}(y)))^2:x \sim_{\overline{G_n'}} y\text{ and }l^{-1}(x)\sim_{\overline{G_n}} l^{-1}(y)\}\\
 &+ \sum\{(f(l^{-1}(x)) - f(l^{-1}(y)))^2:x \sim_{\overline{G_n'}} y\text{ and }l^{-1}(x) \not\sim_{\overline{G_n}} l^{-1}(y)\}\\
 &\leq \sum\{(f(l^{-1}(x)) - f(l^{-1}(y)))^2:x \sim_{\overline{G_n'}} y\text{ and }l^{-1}(x)\sim_{\overline{G_n}} l^{-1}(y)\}\\
 &+ \sum\{f^2(l^{-1}(x)) + f^2(l^{-1}(y)):x \sim_{\overline{G_n'}} y\text{ and }l^{-1}(x) \not\sim_{\overline{G_n}} l^{-1}(y)\}\\
 &= \sum\{(f(x)-f(y))^2:x \sim_{\overline{G_n}} y\}\\
 &= \lam_D(G_n)
\end{align*}
Since $G_n$ was a minimizing subgraph, we in fact have $\lam_D(G'_n) = \lam_D(G_n)$. Hence, we can freely translate path components of $G_n$ without changing $\lam_D(G_n)$. Since we can translate the various path components of $G_n$ until they are all connected, without loss of generality we will suppose that $G_n$ is connected.

Now choose an arbitrary boundary point $\overline{x}$ of $G_n$, and set $G_{n+1} = G_n \cup \{\overline{x}\}$. We claim that $\lam_D(G_{n+1}) < \lam_D(G_n)$.

To see this, let $f$ be a principle eigenfunction of $G_n$. Let $f'$ be the extension of $f$ to $G_{n+1}$ obtained by setting $f'(\overline{x}) = 0$. Then
\begin{align*}
\lam_D(G_{n+1}) &\leq \sum\{(f'(x) - f'(y))^2 : x \sim_{\overline{G_{n+1}}} y\}\\
&= \sum\{(f(x) - f(y))^2 : x \sim_{\overline{G_n}} y\}\\
&= \lam_D(G_n)
\end{align*}
This establishes $\lam_D(G_{n+1}) \leq \lam_D(G_n)$. Now suppose that $\lam_D(G_{n+1}) = \lam_D(G_n)$. Then $R_{G_{n+1}}$ achieves its minimum at $f'$. Hence, Theorem \ref{rayleighQuotientD} implies that $f'$ is an eigenfunction of $G_{n+1}$. However, since $G_{n+1}$ is connected, Corollary \ref{eig} implies that $f'(\overline{x}) > 0$. This is a contradiction, and we conclude that $\lam_D(G_{n+1}) < \lam_D(G_n)$. Now we are done since
\[\lam_D^{(n+1)} \leq \lam_D(G_{n+1}) < \lam_D(G_n) = \lam_D^{(n)}\]
\end{proof}
\begin{prop} \label{eigenConverge}$\lam_D^{(n)} \to 0$ as $n\to\infty$
\end{prop}
\begin{proof}For $n = k^2$, let $S_n$ denote the square subgraph on $n$ vertices. Define $g: \overline{S_n} \to \RR$ to be identically $1$ on $S_n$. Then Theorem \ref{rayleighQuotientD} implies that
\[\lam_D(S_n) \leq R_{S_n}(g) = 4k/n = 4/k \to 0 \text{ as }k \to \infty\]
Then Proposition \ref{eigenDecrease} finishes the proof.
\end{proof}
\begin{prop}Let $G_n$ be any minimizing subgraph on $n$ vertices. Then $G_n$ is connected.
\end{prop}
\begin{proof} Let $f$ be a principle eigenfunction for $G_n$. Since $f$ is nonzero, we can find some connected component of $G_n$ where $f$ is nonzero. Let this component be $V$. Note that $L_D(V)(f) = \lam_D(G_n)f$, i.e. $\lam_D(G_n)$ is an eigenfunction of $V$. Set $k = |V|$. Then $\lam_D^{(k)} \leq \lam_D(V) \leq \lam_D(G_n) = \lam_D^{(n)}$. Now Proposition \ref{eigenDecrease} implies that $k=n$, i.e. $G_n$ must be connected.
\end{proof}

Theorem \ref{mainTheo} only constrains the geometry of large minimizing graphs.  We will now present two examples of subgraphs which illustrate some of the complications involved in attempting to remove the asymptotic nature of Theorem \ref{mainTheo}.

Let $G$ denote
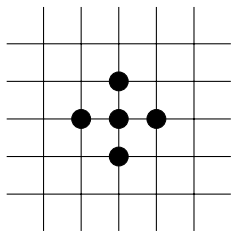
\begin{center}
\begin{tikzpicture}
[interior/.style={circle,draw=black,fill=black, inner sep=0pt,minimum size = 2.5mm},
 boundary/.style={circle,draw=black,fill=black!60, inner sep=0pt,minimum size = 2.5mm},
 exterior/.style={circle,draw=black,fill=white, inner sep=0pt,minimum size = 2.5mm}]
\draw[step=.5cm] (-1.49,-1.49) grid (1.49,1.49);
\draw (0,0) node [interior]{};
\draw (0,.5) node [interior]{};
\draw (0,-.5) node [interior]{};
\draw (.5,0) node [interior]{};
\draw (-.5,0) node [interior]{};
\end{tikzpicture}
\end{center}
and $H$ denote
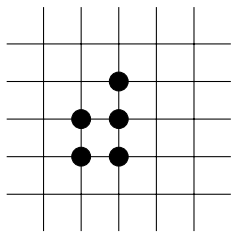
\begin{center}
\begin{tikzpicture}
[interior/.style={circle,draw=black,fill=black, inner sep=0pt,minimum size = 2.5mm},
 boundary/.style={circle,draw=black,fill=black!60, inner sep=0pt,minimum size = 2.5mm},
 exterior/.style={circle,draw=black,fill=white, inner sep=0pt,minimum size = 2.5mm}]
\draw[step=.5cm] (-1.49,-1.49) grid (1.49,1.49);
\draw (-.5,-.5) node [interior]{};
\draw (-.5,0) node [interior]{};
\draw (0,-.5) node [interior]{};
\draw (0,0) node [interior]{};
\draw (0,.5) node [interior]{};
\end{tikzpicture}
\end{center}
The subgraph $G$ appears more symmetric than $H$. Hence, by analogy with the continuous Faber-Krahn inequality, it is natural to conjecture that $\lam_D(G) \leq \lam_D(H)$. However, we will show that $\lam_D(H) < \lam_D(G)$.
Let $f$ be a principle eigenfunction for $G$. Set $x$ to be the the middle point and $y_1$, $y_2$, $y_3$, and $y_4$ to be the other points. Rotating by 90 degrees gives an automorphism of $G$. Hence, Corollary \ref{eigenAuto} implies that $f(y_1) = f(y_2) = f(y_3) = f(y_4)$. Set $y = y_1$. The eigenvalue equations for $f$ reduce to
\[(4-\lam_D(G))f(x) = 4f(y)\]
\[(4-\lam_D(G))f(y) = f(x)\]
Plugging the second equation into the first gives
\[(4-\lam_D(G))^2f(y) = 4f(y)\]
Canceling $f(y)$ and simplifying gives
\[\lam_D(G)^2 - 8\lam_D(G) + 12 = 0\]
The roots of this are $6$ and $2$. Hence $\lam_D(G) = 2$.

Let $V$ be the square subgraph
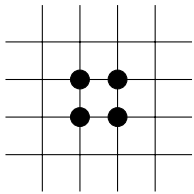
\begin{center}
\begin{tikzpicture}
[interior/.style={circle,draw=black,fill=black, inner sep=0pt,minimum size = 2.5mm},
 boundary/.style={circle,draw=black,fill=black!60, inner sep=0pt,minimum size = 2.5mm},
 exterior/.style={circle,draw=black,fill=white, inner sep=0pt,minimum size = 2.5mm}]
\draw[step=.5cm] (-1.49,-1.49) grid (.99,.99);
\draw (-.5,-.5) node [interior]{};
\draw (-.5,0) node [interior]{};
\draw (0,-.5) node [interior]{};
\draw (0,0) node [interior]{};
\end{tikzpicture}
\end{center}
Set $g$ to be a principle eigenfunction for $V$. Corollary \ref{eigenAuto} implies that $g$ is constant. The relevant equation is
\[(4-\lam_D(V))g = 2g \Rightarrow \lam_D(V) = 2\]
In the proof of Proposition \ref{eigenDecrease}, it was shown that adjoining any boundary point to a connected subgraph will strictly lower $\lam_D$. Hence $\lam_D(H) < \lam_D(V) = 2 = \lam_D(G)$. Thus we see that $\lam_D$ should \emph{not} be thought of as purely measuring the ``symmetry'' or number of automorphisms of a domain.

Now we will start the proof of Theorem \ref{mainTheo}.
\section{Symmetrization}
As noted in the introduction, Steiner symmetrization has proved to be a useful tool when studying Laplacian eigenvalue problems. Suppose $(\Omega,f)$ is a pair consisting of a smooth open domain $\Omega \subset \RR^d$ and a non-negative smooth function $f: \Omega \to \RR$. Then Steiner Symmetrization associates $(\Omega,f)$ to another pair $(\Omega^{\star},f^{\star})$ such that $\Omega^{\star}$ is symmetric with respect to some given hyperplane. The three most important properties which make this useful for eigenvalue problems are
\begin{enumerate}
    \item $\vol(\Omega^{\star}) = \vol(\Omega)$
    \item For any measurable function $\varphi:\RR \to \RR$, e.g. $\varphi = |\cdot|^2$, we have $\int_{\Omega}\varphi(f(x))\ dx = \int_{\Omega^{\star}} \varphi(f^{\star}(x))\ dx$
    \item $\int_{\Omega} |\nabla f(x)|^2\ dx \geq \int_{\Omega^{\star}} |\nabla f^{\star}(x)|^2\ dx$
\end{enumerate}
By taking $f$ to be the relevant eigenfunction and using the Rayleigh quotient as in the proof of the Faber-Krahn inequality, these properties are exactly what is needed to conclude that $\lam(\Omega^{\star}) \leq \lam(\Omega)$. See \cite{n1} for more details.

Now let us consider connected subgraphs of $\ZZ^2$. If we are given a connected subgraph $G$ and a positive function $f$ on $G$, we would like to associate $(G,f)$ to another pair $(G^{\star},f^{\star})$ so that $G^{\star}$ has gained some symmetry, and the analogue of the above properties hold:
\begin{enumerate}
    \item $|G| = |G^{\star}|$
    \item For any function $\varphi: \RR \to \RR$ we have $\sum_{x \in G}\varphi(f(x)) = \sum_{x \in G^{\star}}\varphi(f^{\star}(x))$, i.e. the values of $f^{\star}$ on $G^{\star}$ are a permutation of the values of $f$ on $G$
    \item $\sum\{(f(x)-f(y))^2 : x \sim_G y\} \geq \sum\{(f^{\star}(x)-f^{\star}(y))^2: x \sim_{G^{\star}} y\}$
\end{enumerate}
Given this, the same proof as in the continuous case implies that $\lam_D(G^{\star}) \leq \lam_D(G)$.
\subsection{Horizontal and Vertical Symmetrization}

This paper will employ two different types of discrete symmetrizations. The first will produce graphs ``almost'' symmetric to the $x$ or $y$ axis. The second will produce graphs ``almost'' symmetric to the lines $y = x$ or $y = -x$.

We will start with symmetrization with respect to the $y$ axis, i.e. ``Horizontal Symmetrization''.
\begin{defi}
Let $G$ be a connected subgraph. For any $h \in \ZZ$, the $h$th horizontal slice of $G$ is
\[G^h \equiv \{(x,h) \in G: x \in \ZZ\}\]
\end{defi}
In the subgraph below we have highlighted the $0$th horizontal slice.
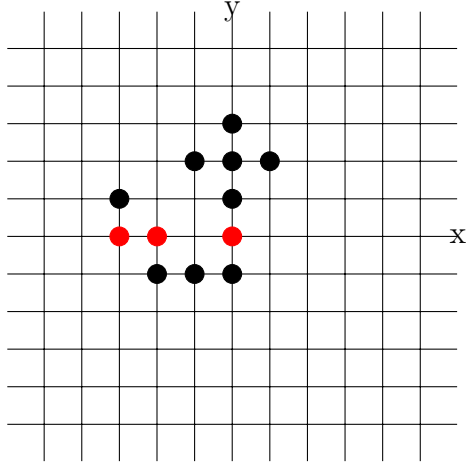
\begin{center}
\begin{tikzpicture}
[interior/.style={circle,draw=black,fill=black, inner sep=0pt,minimum size = 2.5mm},
 boundary/.style={circle,draw=black,fill=black!60, inner sep=0pt,minimum size = 2.5mm},
 exterior/.style={circle,draw=black,fill=white, inner sep=0pt,minimum size = 2.5mm},
 highlight/.style = {circle,draw= red, fill = red, inner sep=0pt, minimum size = 2.5mm}]
\draw[step=.5cm] (-2.99,-2.99) grid (2.99,2.99);
\draw (0,3) node {y};
\draw (3,0) node {x};
\draw (0,0) node [highlight]{};
\draw (0,-.5) node [interior]{};
\draw (0,.5) node [interior]{};
\draw (0,1) node [interior]{};
\draw (0,1.5) node [interior]{};
\draw (.5,1) node [interior]{};
\draw (-.5,1) node [interior]{};
\draw (-.5,-.5) node [interior]{};
\draw (-1,0) node [highlight]{};
\draw (-1.5,0) node [highlight]{};
\draw (-1.5,.5) node [interior]{};
\draw (-1,-.5) node [interior]{};
\end{tikzpicture}
\end{center}
\begin{defi}
An $h$ horizontal path of length $k$ is any subgraph with vertices $(x_0,h)$, $(x_0+1,h)$, $\cdots$, $(x_0+k-1,h)$ for some $x_0$.
\end{defi}
\begin{defi}Let $\chi: \ZZ^2 \to \ZZ^2$ be the automorphism of $\ZZ^2$ defined by $(x,y) \mapsto (-x,y)$. We say a subgraph $L$ is a ``positively centered $h$ horizontal path of length $k$'' if it is an $h$ horizontal line of length $k$ such that either $\chi(L) = L$ or $\chi(L - \{(x_0+k-1,h)\}) = L-\{(x_0+k-1,h)\}$. More informally, up to an extra vertex on the right side, $L$ is symmetric with respect to the $y$ axis.
\end{defi}
To construct $G^{\star}$ from $G$, we take each $G^h$ and replace it with a positively centered $h$ horizontal path of length $|G^h|$. The union of these new slices is $G^{\star}$. We call $G^{\star}$ the ``positive horizontal symmetrization of $G$''
Here we show the positive horizontal symmetrization of the subgraph above.
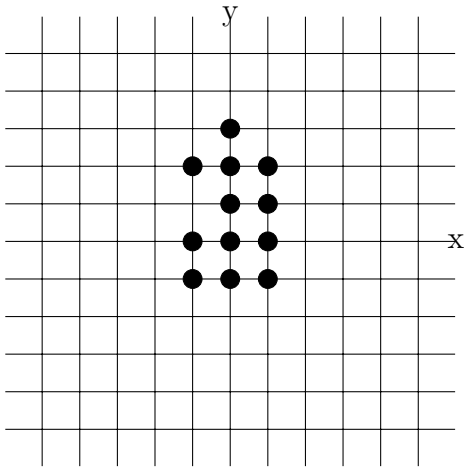
\begin{center}
\begin{tikzpicture}
[interior/.style={circle,draw=black,fill=black, inner sep=0pt,minimum size = 2.5mm},
 boundary/.style={circle,draw=black,fill=black!60, inner sep=0pt,minimum size = 2.5mm},
 exterior/.style={circle,draw=black,fill=white, inner sep=0pt,minimum size = 2.5mm},
 highlight/.style = {circle,draw= red, fill = red, inner sep=0pt, minimum size = 2.5mm}]
\draw[step=.5cm] (-2.99,-2.99) grid (2.99,2.99);
\draw (0,3) node {y};
\draw (3,0) node {x};
\draw (0,-.5) node [interior]{};
\draw (.5,-.5) node [interior]{};
\draw (-.5,0) node [interior]{};
\draw (0,0) node [interior]{};
\draw (.5,0) node [interior]{};
\draw (0,.5) node [interior]{};
\draw (.5,.5) node [interior]{};
\draw (-.5,1) node [interior]{};
\draw (0,1) node [interior]{};
\draw (.5,1) node [interior]{};
\draw (0,1.5) node [interior]{};
\draw (-.5,-.5) node [interior]{};
\end{tikzpicture}
\end{center}
 To any function $f: G \to \RR$ we will associate a function $f^{\star}:G^{\star} \to \RR$ in the following fashion: Fix $h \in \ZZ$. Let $r_1 \geq r_2 \geq \cdots \geq r_{|G^h|}$ be a listing of $\{f(x,h)\}$. Set $f^{\star}(0,h) = r_1$, $f^{\star}(1,h) = r_2$, $f^{\star}(-1,h) = r_3$, $f^{\star}(2,h) = r_4$, $f^{\star}(-2,h) = r_5$, etc. Repeating this process over all horizontal slices defines $f^{\star}$.

It is intuitively plausible that $G^{\star}$ is more ``securely connected'' than $G$. Indeed we have
\begin{theo}\label{symmTheo}
\[\lam_D(G^{\star}) \leq \lam_D(G)\]
\end{theo}
To prove the theorem it is sufficient to show that $R_{G^{\star}}(f^{\star}) \leq R_G(f)$ for any function on $\overline{G}$ that is positive on $G$ and vanishes on $\pa G$. To establish this, we will break up the Rayleigh quotient into many pieces and show the inequality on each piece. First, we note that the values of $f^{\star}$ are permutations of the values of $f$. Thus, we automatically have that $\sum_g f^{\star}(g)^2 = \sum_g f(g)^2$. Therefore, there is no harm in assuming that $\sum_g f(g)^2 = \sum_g f^{\star}(g)^2 = 1$.
\begin{defi}For any $k \in \ZZ$, we define the $k$th horizontal Rayleigh quotient for a function $f$ on $\overline{G}$, by first extending $f$ to be $0$ anywhere it is not defined, and then setting
\[H_k(f) \equiv \sum_{j = -\infty}^{\infty}(f(j+1,k)-f(j,k))^2\]
\end{defi}
\begin{defi}For any $k \in \ZZ$, we define the $k$th vertical Rayleigh quotient for a function $f$ on $\overline{G}$, by first extending $f$ to be $0$ anywhere it is not defined, and then setting
\[V_k(f) \equiv \sum_{j=-\infty}^{\infty}(f(j,k+1)-f(j,k))^2\]
\end{defi}
After extending $f$ to be $0$ everywhere it is not defined, we have
\begin{align*}
R_G(f) &= \sum_{x,y = -\infty}^{\infty}(f(x+1,y) - f(x,y))^2 + (f(x,y+1)-f(x,y))^2\\
       &= \sum_{y =-\infty}^{\infty}\sum_{x = -\infty}^{\infty} (f(x+1,y)-f(x,y))^2\\
       &\ + \sum_{y = -\infty}^{\infty}\sum_{x = -\infty}^{\infty}(f(x,y+1)-f(x,y))^2\\
       &= \sum_{y=-\infty}^{\infty}H_y(f) + \sum_{y=-\infty}^{\infty}V_y(f)
\end{align*}
Hence, Theorem \ref{symmTheo} will follow if we show that $H_k(f^{\star}) \leq H_k(f)$ and $V_k(f^{\star}) \leq V_k(f)$ for all $k$.
To prove this, we need a couple of combinatorial lemmas.

First we establish some notation. Let $x_1 \leq x_2 \leq \cdots \leq x_n$ be a collection of non-negative real numbers. Let $e_0$ and $e_{n+1}$ be two non-negative real numbers with $e_0,e_{n+1} \leq x_i$ for all $i$. The reader should keep in mind the case where $e_0 = e_{n=1} = 0$ and the $x_i's$ are the values of $f$ along a horizontal slice. Let the graph $P_n$ consist of $n+2$ vertices all connected in a line. We label the vertices left to right by $0$, $1$, $\cdots$, $n+1$. We are \emph{not} considering $P_n$ as a subgraph of any larger graph. The permutation group on $n$ letters will be denoted by $S_n$. For any $I \in S_n$, write $I = (i_1,\ i_2,\ \cdots,\ i_n)$, and associate it to a function $f_I$ on $P_n$ defined by
\begin{equation*}
    f_I(p) \equiv \left\{
    \begin{array}{rl}
    e_0 & \text{if } p = 0\\
    x_{i_p} & \text{if } p \in [1,n]\\
    e_{n+1} & \text{if } p = n+1
    \end{array} \right.
\end{equation*}
Now define a map $\hat{R}:S_n \to \RR$ by sending $I$ to the numerator of the Rayleigh quotient of $f_I$:
\[I \mapsto (e_0-x_{i_1})^2 + (x_{i_1}-x_{i_2})^2 + \cdots + (x_{i_{n-1}}-x_{i_n})^2 + (x_{i_n}-e_{n+1})^2\]

Our proof of Theorem \ref{symmTheo} will rely on finding which $I \in S_n$ minimize $\hat{R}$. Towards this goal we now define some candidate minimizers $J_R = (j^r_1,\ j^r_2, \ \cdots,\ j^r_n)$ and $J_L = (j^l_1,\ j^l_2,\ \cdots,\ j^l_n)$ by
\begin{enumerate}
    \item $j^r_{\lceil n/2\rceil} = n$, $j^r_{\lceil n/2\rceil + 1} = n-1$, $j^r_{\lceil n/2\rceil - 1} = n-2$, $j^r_{\lceil n/2\rceil + 2} = n-3$, etc.
    \item $j^l_{\lceil n/2\rceil} = n$, $j^l_{\lceil n/2\rceil - 1} = n-1$, $j^l_{\lceil n/2\rceil + 1} = n-2$, $j^l_{\lceil n/2\rceil - 2} = n-3$, etc.
\end{enumerate}
The function $f_{J_R}$ has the following values on $P_n$:
\begin{center}
            \begin{tikzpicture}[description/.style={fill=white,inner sep=2pt}]
                    \matrix (m) [matrix of math nodes, row sep=3em,
                    column sep=1.5em, text height=1.5ex, text depth=0.25ex]
                    { e_0 & \cdots & x_{n-4} & x_{n-2} & x_n & x_{n-1} & x_{n-3} & \cdots & e_{n+1} \\};
                    \path[-,font=\scriptsize]
                    (m-1-1) edge node[auto] {} (m-1-2)
                    (m-1-2) edge node[auto] {} (m-1-3)
                    (m-1-3) edge node[auto] {} (m-1-4)
                    (m-1-4) edge node[auto] {} (m-1-5)
                    (m-1-5) edge node[auto] {} (m-1-6)
                    (m-1-6) edge node[auto] {} (m-1-7)
                    (m-1-7) edge node[auto] {} (m-1-8)
                    (m-1-8) edge node[auto] {} (m-1-9);
            \end{tikzpicture}
\end{center}
Note that this corresponds to the permutation associated to the restriction of $f^{\star}$ to a fixed horizontal slice.
\begin{lemm}\label{horizontalLemm} $\hat{R}$ achieves its minimum value at both $J_R$ and $J_L$.
\end{lemm}
\begin{proof}We use induction on $n$. The base case is trivial so let us assume the lemma holds for $n=1$. Let $J_R^{(n-1)}$ and $J_L^{(n-1)}$ be the minimizers associated to $x_1 \leq x_2 \leq \cdots \leq x_{n-1}$. The proofs for $J_R$ and $J_L$ are symmetric so we will just consider the $J_R$ case. Let $I \in S_n$. We need to show that $\hat{R}(I) - \hat{R}(J_R) \geq 0$.
\begin{align*}
\hat{R}(J_R) &= (e_0-x_{j^r_1})^2 + (x_{j^r_1} - x_{j^r_2})^2 + \cdots + (x_{n-4} - x_{n-2})^2\\
       &\ + (x_{n-2} - x_n)^2 + (x_n-x_{n-1})^2 + (x_{n-1}-x_{n-3})^2 + \cdots \\
       &\ + (x_{j^r_{n-1}}-x_{j^r_n})^2 + (x_{j_n^r}-e_{n+1})^2\\
       &= \Big[(e_0-x_{j_1^r})^2 + (x_{j_1^r} - x_{j_2^r})^2 + \cdots +(x_{n-4}-x_{n-2})^2\\
       &\ + (x_{n-1}-x_{n-2})^2 + (x_{n-1}-x_{n-3})^2 + \cdots + (x_{j^r_{n-1}}-x_{j_n^r})^2\\
       &\ + (x_{j_n^r}-e_{n+1})^2\Big] - (x_{n-1}-x_{n-2})^2 + (x_n-x_{n-2})^2 + (x_n-x_{n-1})^2
\end{align*}
The bracketed terms are exactly $\hat{R}(J_L^{(n-1)})$. To see this, simply consider the relevant diagram for $J_L^{(n-1)}$:
\begin{center}
            \begin{tikzpicture}[description/.style={fill=white,inner sep=2pt}]
                    \matrix (m) [matrix of math nodes, row sep=3em,
                    column sep=1.5em, text height=1.5ex, text depth=0.25ex]
                    { e_0 & \cdots & x_{n-4} & x_{n-2} & x_{n-1} & x_{n-3} & x_{n-5} & \cdots & e_{n+1} \\};
                    \path[-,font=\scriptsize]
                    (m-1-1) edge node[auto] {} (m-1-2)
                    (m-1-2) edge node[auto] {} (m-1-3)
                    (m-1-3) edge node[auto] {} (m-1-4)
                    (m-1-4) edge node[auto] {} (m-1-5)
                    (m-1-5) edge node[auto] {} (m-1-6)
                    (m-1-6) edge node[auto] {} (m-1-7)
                    (m-1-7) edge node[auto] {} (m-1-8)
                    (m-1-8) edge node[auto] {} (m-1-9);
            \end{tikzpicture}
\end{center}
Thus we conclude that
\[\hat{R}(J_R) = \hat{R}(J_L^{(n-1)}) - (x_{n-1}-x_{n-2})^2 + (x_n-x_{n-2})^2 + (x_n-x_{n-1})^2\]
A completely analogous argument implies that
\[\hat{R}(I) = \hat{R}\left(I'\right) + (x_n-a)^2 + (x_n-b)^2 - (a-b)^2\]
where $I' \in S_{n-1}$ is associated to some permutation of $x_1 \leq \cdots \leq x_{n-1}$ and $a,b$ are non-negative real numbers with $a \leq x_{n-1}$ and $b \leq x_{n-2}$. Now
\begin{align*}
\hat{R}(I) - \hat{R}(J_R) &= \hat{R}(I') - \hat{R}(J_L^{(n-1)}) - (x_n-x_{n-1})^2 - (x_n-x_{n-2})^2\\
              &\ + (x_{n-1}-x_{n-2})^2 + (x_n-a)^2 + (x_n-b)^2 - (a-b)^2\\
              &\geq -(x_n-x_{n-1})^2 - (x_n-x_{n-2})^2 + (x_{n-1}-x_{n-2})^2\\
              &\ + (x_n-a)^2 + (x_n-b)^2 - (a-b)^2\\
              &= 2x_nx_{n-1}+ 2x_nx_{n-2} - 2x_{n-1}x_{n-2}\\
              &\ -2x_na - 2x_nb + 2ab\\
              &= 2[(x_n-b)(x_{n-1}-a)+(x_n-x_{n-1})(x_{n-2}-b)]\\
              &\geq 0
\end{align*}
\end{proof}
This implies
\begin{lemm}Let $f$ be a function on $\overline{G}$ that is positive on $G$ and $0$ on $\pa G$. Then $H_k(f^{\star}) \leq H_k(f)$.
\end{lemm}
\begin{proof}For some large $N$, $H_k(f^{\star}) = \sum_{j=-N}^N(f^{\star}(j+1,k)-f^{\star}(j,k))^2$, $H_k(f) = \sum_{j=-N}^N(f(j+1,k)-f(j,k))^2$, and $f(-N,k) = f(N,k) = f^{\star}(-N,k) = f^{\star}(N,k) = 0$. The values of $f^{\star}$ along $\{(j,k)\}_{j=-N}^N$ are a permutation of the values of $f$ along $\{(j,k)\}_{j=-N}^N$. Lemma \ref{horizontalLemm} applies with $n = 2N+1$. The minimizing permutation $J^{(2N+1)}_R$ exactly corresponds to the restriction of $f^{\star}$ to $\{(j,k)\}_{j=-N}^N$.
\end{proof}
Now we need a combinatorial lemma to handle the vertical Rayleigh quotient. As before, we first establish some notation. Let $x_1\leq x_2\leq \cdots \leq x_n$ and $y_1 \leq y_2 \leq \cdots \leq y_n$ be fixed collections of non-negative real numbers. For any $I \in S_n$ we have an associated diagram
\begin{center}
            \begin{tikzpicture}[description/.style={fill=white,inner sep=2pt}]
                    \matrix (m) [matrix of math nodes, row sep=3em,
                    column sep=2.5em, text height=1.5ex, text depth=0.25ex]
                    { x_1     & x_2     & \cdots & x_{n-1}     & x_n \\
                      y_{i_1} & y_{i_2} & \cdots & y_{i_{n-1}} & y_{i_n}\\};
                    \path[-,font=\scriptsize]
                    (m-1-1) edge node[auto] {} (m-2-1)
                    (m-1-2) edge node[auto] {} (m-2-2)
                    (m-1-3) edge node[auto] {} (m-2-3)
                    (m-1-4) edge node[auto] {} (m-2-4)
                    (m-1-5) edge node[auto] {} (m-2-5);
            \end{tikzpicture}
\end{center}
Define a function $\tilde{R}: S_n \to \RR$ as the numerator of the Rayleigh quotient of the above diagram:
\[I \mapsto (x_1-y_{i_1})^2 + (x_2 - y_{i_2})^2 + \cdots + (x_n - y_{i_n})^2\]
\begin{lemm}\label{verticalLemm}$\tilde{R}$ achieves it minimum value at $I^n = (1,\ 2,\ \cdots,\ n)$.
\end{lemm}
\begin{proof}We will use induction on $n$. The claim is immediate for $n=1$ so let us assume that the lemma has been proven for $n-1$. Let $I^{(n-1)} \in S_{n-1}$ be given by $I^{(n-1)} = (1,\ 2,\ \cdots,\ n-1)$. Choose any $J \in S_n$. We need to show that $\tilde{R}(J) - \tilde{R}(I^n) \geq 0$. If $J = (\hat{J},\ n)$, then $\tilde{R}(J) - \tilde{R}\left(I^n\right) = \tilde{R}\left(\hat{J}\right) - \tilde{R}\left(I^{(n-1)}\right)$, and an application of the induction hypothesis gives $\tilde{R}(J) - \tilde{R}\left(I^n\right) \geq 0$. So assume that $j_n \neq n$, and fix $a$ such that $j_a = n$.
We have
\begin{align*}
\tilde{R}(J) &= (x_1-y_{j_1})^2 + \cdots +(x_a-y_n)^2 + \cdots + (x_n-y_{j_n})^2\\
     &= (x_1 -y_{j_1})^2 + \cdots + (x_{a-1}-y_{j_{a-1}})^2 + (x_a - y_{j_n})^2\\
     &\ - (x_a - y_{j_n})^2 + (x_a-y_n)^2 + (x_{a+1}-y_{j_{a+1}})^2 + \cdots \\
     &\ + (x_{n-1}-y_{j_{n-1}})^2 + (x_n-y_{j_n})^2 \\
     &= \Big[(x_1 -y_{j_1})^2 + \cdots + (x_{a-1}-y_{j_{a-1}})^2 + (x_a - y_{j_n})^2\\
     &\ + (x_{a+1}-y_{j_{a+1}})^2 + \cdots + (x_{n-1}-y_{j_{n-1}})^2\Big]\\
     &\ - (x_a - y_{j_n})^2 + (x_a-y_n)^2 + (x_n-y_{j_n})^2
\end{align*}
Let $J' = (j_1,\ j_2,\ \cdots, j_{a-1},\ j_n,\ j_{a+1},\ \cdots,\ j_{n-1}) \in S_{n-1}$. Then the bracketed terms are equal to $\tilde{R}(J')$. Hence
\[\tilde{R}(J) = \tilde{R}(J') - (x_a - y_{j_n})^2 + (x_a-y_n)^2 + (x_n-y_{j_n})^2\]

We also have $\tilde{R}\left(I^n\right) = \tilde{R}\left(I^{(n-1)}\right) + (x_n-y_n)^2$. Putting this together gives
\begin{align*}
\tilde{R}(J) - \tilde{R}\left(I^n\right) &= \tilde{R}\left(J'\right)- \tilde{R}\left(I^{(n-1)}\right) - (x_a-y_{j_n})^2\\
              &\ + (x_a - y_n)^2 + (x_n-y_{j_n})^2 - (x_n-y_n)^2\\
              &\geq (x_a - y_n)^2 + (x_n-y_{j_n})^2 - (x_a-y_{j_n})^2 - (x_n-y_n)^2\\
              &= -2x_ay_n - 2x_ny_{j_n} + 2x_ay_{j_n} + 2x_ny_n\\
              &= 2(x_n-x_a)(y_n-y_{j_n})\\
              &\geq 0
\end{align*}
\end{proof}
\begin{lemm}Let $f$ be a function on $\overline{G}$ that is positive on $G$ and $0$ on $\pa G$. Then $V_k(f^{\star}) \leq V_k(f)$.
\end{lemm}
\begin{proof}Choose $N$ large enough so that $V_k(f^{\star}) = \sum_{k=-N}^N(f^{\star}(j,k+1)-f^{\star}(j,k))^2$ and $V_k(f) = \sum_{k=-N}^N(f(j,k+1)-f(j,k))^2$. The values of $f^{\star}$ along $G^k$ and $G^{k-1}$ are just permutations of the values of $f$ along $G^k$ and $G^{k-1}$. Hence Lemma \ref{verticalLemm} is applicable. The vertical Rayleigh quotient $V_k(f)$ is equal to $R(J)$ for some $J \in S_n$. Next we see that $V_k(f^{\star})$ pairs the greatest values of the two slices together, the next two greatest together, etc. Thus $V_k(f^{\star}) = R(I^n)$. Then Lemma \ref{verticalLemm} implies $V_k(f^{\star}) \leq V_k(f)$.
\end{proof}
This concludes the proof of Theorem \ref{symmTheo}.

As is, Theorem \ref{symmTheo} is not terribly useful because it does not produce strict inequalities. Thus, it will never show that a subgraph is not a minimizing subgraph. Next, we characterize some classes of subgraphs where the inequality in Theorem \ref{symmTheo} is strict.

Our first such result is
\begin{theo}\label{symmTheoStrict}Suppose that $G$ is a connected subgraph such for some $h$, the $h$ horizontal slice is non-empty and disconnected. Then
\[\lam_D(G^{\star}) < \lam_D(G)\]
\end{theo}
As with the proof of Theorem \ref{symmTheo}, the proof of Theorem \ref{symmTheoStrict} relies on a purely combinatorial lemma. We use the notation from Lemma \ref{horizontalLemm}. Recall that
\begin{enumerate}
\item $R_G(f) = \sum_{k=-\infty}^{\infty}H_k(f) + \sum_{k=-\infty}^{\infty}V_k(f)$
\item $H_k(f^{\star}) \leq H_k(f)\ \forall k$
\item $V_k(f^{\star}) \leq V_k(f)\ \forall k$
\end{enumerate}
Hence, the theorem will follow if we establish $H_h(f^{\star}) < H_h(f)$.
\begin{lemm}\label{strictHorizontalLemm} Suppose  $I \in S_n$ such that for some $l < k$; $x_{i_l} > x_{i_k}$ and $x_{i_{l-1}} < x_{i_{k+1}}$. Then $\hat{R}$ does not achieve its minimum value at $I$.
\end{lemm}
\begin{proof}The diagram for $I$ is
\begin{center}
            \begin{tikzpicture}[description/.style={fill=white,inner sep=2pt}]
                    \matrix (m) [matrix of math nodes, row sep=3em,
                    column sep=2.5em, text height=1.5ex, text depth=0.25ex]
                    { x_{i_1} & x_{i_2} & \cdots & x_{i_l} & \cdots & x_{i_k} & \cdots & x_{i_n} \\};
                    \path[-,font=\scriptsize]
                    (m-1-1) edge node[auto] {} (m-1-2)
                    (m-1-2) edge node[auto] {} (m-1-3)
                    (m-1-3) edge node[auto] {} (m-1-4)
                    (m-1-4) edge node[auto] {} (m-1-5)
                    (m-1-5) edge node[auto] {} (m-1-6)
                    (m-1-6) edge node[auto] {} (m-1-7)
                    (m-1-7) edge node[auto] {} (m-1-8);
            \end{tikzpicture}
\end{center}
We will produce $I' \in S_n$ by ``flipping'' the path from $x_{i_l}$ to $x_{i_k}$ to produce a diagram
\begin{center}
            \begin{tikzpicture}[description/.style={fill=white,inner sep=2pt}]
                    \matrix (m) [matrix of math nodes, row sep=3em,
                    column sep=1.25em, text height=1.5ex, text depth=0.25ex]
                    { x_{i_1} & \cdots & x_{i_{l-1}} & x_{i_k} & x_{i_{k-1}} & \cdots & x_{i_l} & x_{i_{k+1}} & \cdots & x_{i_n} \\};
                    \path[-,font=\scriptsize]
                    (m-1-1) edge node[auto] {} (m-1-2)
                    (m-1-2) edge node[auto] {} (m-1-3)
                    (m-1-3) edge node[auto] {} (m-1-4)
                    (m-1-4) edge node[auto] {} (m-1-5)
                    (m-1-5) edge node[auto] {} (m-1-6)
                    (m-1-6) edge node[auto] {} (m-1-7)
                    (m-1-7) edge node[auto] {} (m-1-8)
                    (m-1-8) edge node[auto] {} (m-1-9)
                    (m-1-9) edge node[auto] {} (m-1-10);
            \end{tikzpicture}
\end{center}
That is,
\[I' = (i_1,\ i_2,\ \cdots,\ i_{l-1},\ i_k,\ i_{k-1},\ \cdots, i_l,\ i_{k+1},\ \cdots,\ i_n)\]
Then
\begin{align*}
\hat{R}(I) - \hat{R}(I') &= (x_{i_k}-x_{i_{k+1}})^2 + (x_{i_{l-1}}-x_{i_l})^2 - (x_{i_{l-1}}-x_{i_k})^2 - (x_{i_l}-x_{i_{k+1}})^2\\
             &= -2x_{i_k}x_{i_{k+1}} - 2x_{i_{l-1}}x_{i_l} + 2x_{i_{l-1}}x_{i_k} + 2x_{i_l}x_{i_{k+1}}\\
             &= 2(x_{i_l}-x_{i_k})(x_{i_{k+1}}-x_{i_{l-1}})\\
             &> 0
\end{align*}
\end{proof}
Now we return to the case of a subgraph $G$ with a disconnected $h$th horizontal slice. Choose $N_1$ and $N_2$ with the smallest possible magnitude so that $f(x,h) = 0$ for all $x \leq N_1$ and $x \geq N_2$. Then order the values of $\{f(x,h)\}_{x=N_1}^{N_2}$ by $f_1 \leq f_2 \leq \cdots \leq f_{2N+1}$. The function $f$ is associated to a permutation $I \in S_{N_1+N_2+1}$ with $f(x,h) = f_{i_{x+N_1+1}}$. By choice of $N_1$ we have $f_{i_1} = 0$ and $f_{i_2} > 0$. Since the $h$th horizontal slice is not connected, and $f$ is only non-zero on $G$, we have some $k$ so that, $k > 2$, $f_{i_k} = 0$, and $f_{i_{k+1}} > 0$. Thus Lemma \ref{strictHorizontalLemm} is applicable and we conclude that $\hat{R}$ does not achieve its minimum value at $I$. The function $f^{\star}$ is associated to some other permutation $J \in S_{N_1+N_2+1}$. In Lemma \ref{horizontalLemm} we proved that the permutation $J$ minimizes $\hat{R}$. Since $I$ is not a minimum value of $\hat{R}$, we must have $H_h(f^{\star}) < H_h(f)$. This concludes the proof Theorem \ref{symmTheoStrict}.

Next we give another class of graphs whose eigenvalues are strictly lowered by symmetrization. First we need some more definitions.
\begin{defi}\label{wall}Let $U$ and $V$ be connected subgraphs whose points all have $y$ coordinates $n$ and $m$ respectively. We say that $U$ ``vertically walls in'' $V$ if $(x,n) \in V$ implies $(x,m) \in U$.
\end{defi}
In the following subgraph, the $0$th horizontal slice vertically walls in the $1$st horizontal slice.
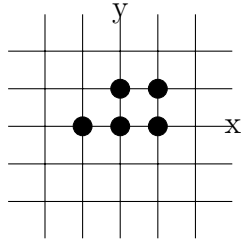
\begin{center}
\begin{tikzpicture}
[interior/.style={circle,draw=black,fill=black, inner sep=0pt,minimum size = 2.5mm},
 boundary/.style={circle,draw=black,fill=black!60, inner sep=0pt,minimum size = 2.5mm},
 exterior/.style={circle,draw=black,fill=white, inner sep=0pt,minimum size = 2.5mm},
 highlight/.style = {circle,draw= red, fill = red, inner sep=0pt, minimum size = 2.5mm}]
\draw[step=.5cm] (-1.49,-1.49) grid (1.49,1.49);
\draw (0,1.5) node {y};
\draw (1.5,0) node {x};
\draw (0,0) node [interior]{};
\draw (-.5,0) node [interior]{};
\draw (.5,0) node [interior]{};
\draw (0,.5) node [interior]{};
\draw (.5,.5) node [interior]{};
\end{tikzpicture}
\end{center}
In this subgraph, neither slice vertically walls in the other.
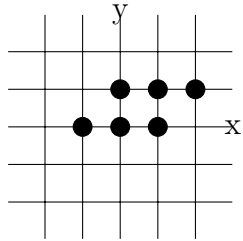
\begin{center}
\begin{tikzpicture}
[interior/.style={circle,draw=black,fill=black, inner sep=0pt,minimum size = 2.5mm},
 boundary/.style={circle,draw=black,fill=black!60, inner sep=0pt,minimum size = 2.5mm},
 exterior/.style={circle,draw=black,fill=white, inner sep=0pt,minimum size = 2.5mm},
 highlight/.style = {circle,draw= red, fill = red, inner sep=0pt, minimum size = 2.5mm}]
\draw[step=.5cm] (-1.49,-1.49) grid (1.49,1.49);
\draw (0,1.5) node {y};
\draw (1.5,0) node {x};
\draw (0,0) node [interior]{};
\draw (-.5,0) node [interior]{};
\draw (.5,0) node [interior]{};
\draw (0,.5) node [interior]{};
\draw (.5,.5) node [interior]{};
\draw (1,.5) node [interior]{};
\end{tikzpicture}
\end{center}
There are corresponding notions for vertical slices.
Our final horizontal symmetrization theorem is
\begin{theo}\label{symmTheoWall}Let $G$ be a subgraph. If there exists $h \in \ZZ$ such that neither the $h$th horizontal slice of $G$ nor the $h+1$st horizontal slice of $G$ wall in each other, then
\[\lam_D(G^{\star}) < \lam_D(G)\]
\end{theo}
For this we need another combinatorial lemma. We use the notation from Lemma \ref{verticalLemm}.
\begin{lemm}\label{strictVerticalLemm}Suppose $I \in S_n$ and there exists $k$ and $l$ with $x_k > x_l$ and $y_{i_l} > y_{i_k}$. Then $\tilde{R}$ does not achieve its minimum value at $I$.
\end{lemm}
\begin{proof}Without loss of generality let $k < l$. Then define a permutation
\[J = (i_1,\ \cdots,\ i_{k-1},\ i_l,\ i_{k+1},\ \cdots,\ i_{l-1},\ i_k,\ i_{l+1},\ \cdots,\ i_n)\]
Then we have
\begin{align*}
\tilde{R}(I) - \tilde{R}(J) &= (x_k-y_{i_k})^2 + (x_l - y_{i_l})^2 - (x_k-y_{i_l})^2 - (x_l-y_{i_k})^2\\
            &= -2x_ky_{i_k} - 2x_ly_{i_l} + 2x_ky_{i_l} + 2x_ly_{i_k}\\
            &= 2(x_k-x_l)(y_{i_l}-y_{i_k})\\
            &> 0
\end{align*}
\end{proof}
\begin{proof}(Theorem \ref{symmTheoWall})
Suppose that neither $G^k$ nor $G^{k-1}$ wall in each other. Let $f$ be a normalized eigenfunction for $G$. We will show that $V_k(f^{\star}) < V_k(f)$. We have
\[V_k(f) = \sum_{j=-\infty}^{\infty}(f(j,k+1)-f(j,k))^2\]
Recall that $f$ is non-zero at a point if and only if the point lies in $G$. Since neither $G^k$ or $G^{k-1}$ wall in each other, we can find $i$ and $j$ so that $f(i,k+1) = 0$, $f(i,k) > 0$, $f(j,k+1) > 0$, and $f(j,k) = 0$. Now Lemma \ref{strictVerticalLemm} applies and the Theorem immediately follows.
\end{proof}

We can also define negative horizontal symmetrization, positive vertical symmetrization, and negative vertical symmetrization. These are completely analogous to positive horizontal symmetrization, and there are corresponding versions of Theorem \ref{symmTheo}, Theorem \ref{symmTheoStrict}, and Theorem \ref{symmTheoWall}.

\subsection{Diagonal Symmetrization}
For diagonal Symmetrization we will use ``diagonal slices'' of our subgraph instead of horizontal or vertical slices. The diagonal slices do not interact as nicely with $\ZZ^2$ and thus the combinatorics involved are a little more subtle.
\begin{defi}Let $G$ be a subgraph. For $h \in \ZZ$, the $h$th diagonal slice of $G$ consist of all points in $G \cap \{y = x + h\}$.
\end{defi}
In the following subgraph we have highlighted the $3$rd diagonal slice.
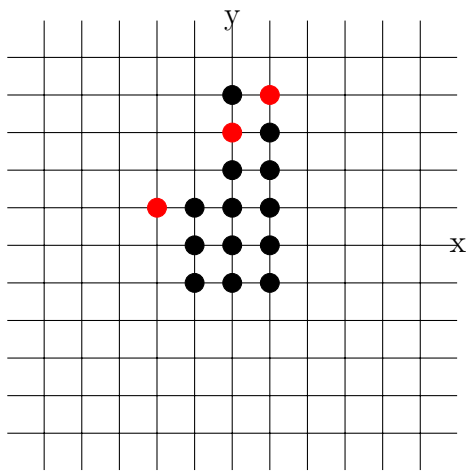
\begin{center}
    \begin{tikzpicture}
        [interior/.style={circle,draw=black,fill=black, inner sep=0pt,minimum size = 2.5mm},
        boundary/.style={circle,draw=black,fill=black!60, inner sep=0pt,minimum size = 2.5mm},
        exterior/.style={circle,draw=black,fill=white, inner sep=0pt,minimum size = 2.5mm},
        highlight/.style = {circle,draw= red, fill = red, inner sep=0pt, minimum size = 2.5mm}]
        \draw[step=.5cm] (-2.99,-2.99) grid (2.99,2.99);
        \draw (0,3) node {y};
        \draw (3,0) node {x};
        \draw (-.5,-.5) node [interior]{};
        \draw (0,-.5) node [interior]{};
        \draw (.5,-.5) node [interior]{};
        \draw (-.5,0) node [interior]{};
        \draw (0,0) node [interior]{};
        \draw (.5,0) node [interior]{};
        \draw (-1,.5) node [highlight]{};
        \draw (-.5,.5) node [interior]{};
        \draw (0,.5) node [interior]{};
        \draw (.5,.5) node [interior]{};
        \draw (0,1) node [interior]{};
        \draw (.5,1) node [interior]{};
        \draw (0,1.5) node [highlight]{};
        \draw (.5,1.5) node [interior]{};
        \draw (0,2) node [interior]{};
        \draw (0.5,2) node [highlight]{};
    \end{tikzpicture}
\end{center}
Diagonal symmetrization will center each of these slices. Consider the line $y = -x$ which cuts $\ZZ^2$ in half.
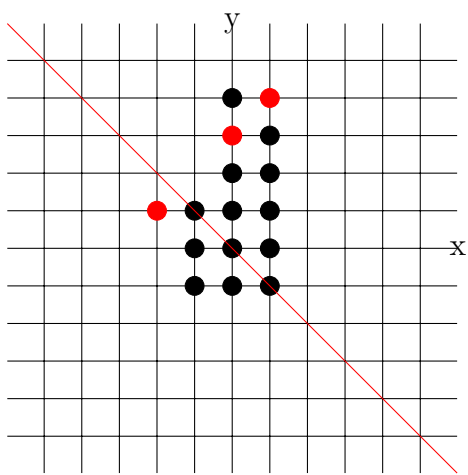
\begin{center}
    \begin{tikzpicture}
        [interior/.style={circle,draw=black,fill=black, inner sep=0pt,minimum size = 2.5mm},
        boundary/.style={circle,draw=black,fill=black!60, inner sep=0pt,minimum size = 2.5mm},
        exterior/.style={circle,draw=black,fill=white, inner sep=0pt,minimum size = 2.5mm},
        highlight/.style = {circle,draw= red, fill = red, inner sep=0pt, minimum size = 2.5mm}]
        \draw[step=.5cm] (-2.99,-2.99) grid (2.99,2.99);
        \draw (0,3) node {y};
        \draw (3,0) node {x};
        \draw (-.5,-.5) node [interior]{};
        \draw (0,-.5) node [interior]{};
        \draw (.5,-.5) node [interior]{};
        \draw (-.5,0) node [interior]{};
        \draw (0,0) node [interior]{};
        \draw (.5,0) node [interior]{};
        \draw (-1,.5) node [highlight]{};
        \draw (-.5,.5) node [interior]{};
        \draw (0,.5) node [interior]{};
        \draw (.5,.5) node [interior]{};
        \draw (0,1) node [interior]{};
        \draw (.5,1) node [interior]{};
        \draw (0,1.5) node [highlight]{};
        \draw (.5,1.5) node [interior]{};
        \draw (0,2) node [interior]{};
        \draw (0.5,2) node [highlight]{};
        \draw [color = red] (2.99,-2.99) -- (-2.99,2.99);
    \end{tikzpicture}
\end{center}
To diagonally symmetrize our graph, we  replace each diagonal slice with a new diagonal slice as symmetric as possible with respect to $y = -x$. If we have an extra point, then we put it on the right. We spare the reader a formal definition of this. The diagonal symmetrization of the above subgraph is
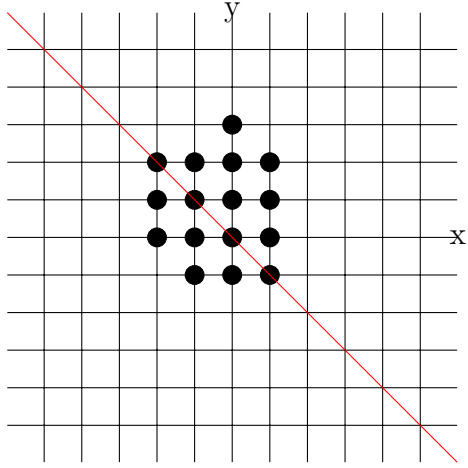
\begin{center}
    \begin{tikzpicture}
        [interior/.style={circle,draw=black,fill=black, inner sep=0pt,minimum size = 2.5mm},
        boundary/.style={circle,draw=black,fill=black!60, inner sep=0pt,minimum size = 2.5mm},
        exterior/.style={circle,draw=black,fill=white, inner sep=0pt,minimum size = 2.5mm},
        highlight/.style = {circle,draw= red, fill = red, inner sep=0pt, minimum size = 2.5mm}]
        \draw[step=.5cm] (-2.99,-2.99) grid (2.99,2.99);
        \draw (0,3) node {y};
        \draw (3,0) node {x};
        \draw (-.5,-.5) node [interior]{};
        \draw (0,-.5) node [interior]{};
        \draw (.5,-.5) node [interior]{};
        \draw (-1,0) node [interior]{};
        \draw (-.5,0) node [interior]{};
        \draw (0,0) node [interior]{};
        \draw (0,1.5) node [interior]{};
        \draw (.5,0) node [interior]{};
        \draw (.5,1) node [interior]{};
        \draw (-1,.5) node [interior]{};
        \draw (-.5,.5) node [interior]{};
        \draw (0,.5) node [interior]{};
        \draw (.5,.5) node [interior]{};
        \draw (-1,1) node [interior]{};
        \draw (-.5,1) node [interior]{};
        \draw (0,1) node [interior]{};
        \draw [color = red] (2.99,-2.99) -- (-2.99,2.99);
    \end{tikzpicture}
\end{center}
We denote the diagonal symmetrization of a subgraph $G$ by $G^{\dagger}$. For any function on $G$, we also get a function $f^{\dagger}$ on $G^{\dagger}$ by mimicking the definition of $f^{\star}$. For each $k$, let $x^{(k)}_1 \geq x^{(k)}_2 \geq \cdots \geq x^{(k)}_n$ be the values of $f$ along the $k$th diagonal slice. Now order the points on the $k$th diagonal slice of $G^{\dagger}$ by their distance from $y = -x$. If two points have the same distance, then the one on the right goes first.  If $z_1 \geq z_2 \geq \cdots \geq z_n$ is the listing of points on the $k$th diagonal slice of $G^{\dagger}$, define $f^{\dagger}(z_j) = x^{(k)}_j$.

As with $G^{\star}$ we have
\begin{theo}\label{diagTheo}
\[\lam_D(G^{\dagger}) \leq \lam_D(G)\]
\end{theo}
We will reuse ideas from the previous section. First we will break up the Rayleigh quotient into terms involving adjacent diagonal slices. Then we will use purely combinatorial methods to show the inequality on each term.
\begin{defi}Let $f: \ZZ^2 \to \RR$ be any function non-zero on finitely many vertices. Then, for $k \in \ZZ$ we define
\[D_k(f) = \sum_{j=-\infty}^{\infty}(f(j,k+j) - f(j+1,k+j))^2 + (f(j,k+j) - f(j,k+j-1))^2\]
We call this the ``diagonal Rayleigh quotient.''
\end{defi}
If $f$ is a normalized eigenfunction for $G$, then after extending $f$ to be $0$ anywhere it is not defined, we have
\[R(f) = \sum_{k=-\infty}^{\infty}D_k(f)\]
Hence, to prove Theorem \ref{diagTheo} we just need to establish that $D_k(f^{\dagger}) \leq D_k(f)$ for all $k$. Now we will recast this into a purely combinatorial question.

Suppose we have collections of non-negative real numbers $x_0 \leq x_1 \leq x_2 \leq \cdots \leq x_n$ and $y_0 \leq y_1 \leq y_2 \leq \cdots \leq y_n$. We refer to $x_0$ and $y_0$ as the endpoints. We associate each pair $(I,J) \in S_n\times S_n$ with the following diagram
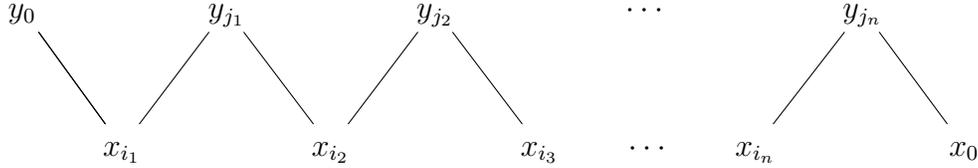
\begin{center}
    \begin{tikzpicture}[description/.style={fill=white,inner sep=2pt}]
                    \matrix (m) [matrix of math nodes, row sep=3em,
                    column sep=1.5em, text height=1.5ex, text depth=0.25ex]
                    {y_0 &         &  y_{j_1} &         & y_{j_2} &         & \cdots &         & y_{j_n} & \\
                       & x_{i_1} &          & x_{i_2} &         & x_{i_3} &\cdots  & x_{i_n} &         & x_0\\};
                    \path[-,font=\scriptsize]
                    (m-1-1) edge node[auto] {} (m-2-2)
                    (m-2-2) edge node[auto] {} (m-1-3)
                    (m-1-3) edge node[auto] {} (m-2-4)
                    (m-2-4) edge node[auto] {} (m-1-5)
                    (m-1-5) edge node[auto] {} (m-2-6)
                    (m-1-1) edge node[auto] {} (m-2-2)
                    (m-2-8) edge node[auto] {} (m-1-9)
                    (m-1-9) edge node[auto] {} (m-2-10);
    \end{tikzpicture}
\end{center}
Now we define a function $\overline{R}:S_n\times S_n \to \RR$ by taking the numerator of the Rayleigh quotient of the above graph. That is,
\[(I,J) \mapsto (x_{i_1}-y_0)^2 + (x_{i_1}-y_{j_1})^2 + (y_{j_1}-x_{i_2})^2 + \cdots + (y_{j_n}-x_{i_n})^2 + (y_{j_n}-x_0)^2\]
We are interested in minimizing $\overline{R}$.
\begin{lemm}\label{diagLemm}$\overline{R}$ achieves its minimum value at $(I,J)$ where $I$ is defined by
\[i_1 = 1,\ i_n = 2,\ i_2 = 3,\ \cdots\]
and $J$ is defined by
\[j_n = 1,\ j_1 = 2,\ j_{n-1} = 3,\ \cdots\]
\end{lemm}
\begin{proof}
Suppose we have $(H,K) \in S_n \times S_n$ with $H = (h_1,\ \cdots,\ h_n)$ and $K = (k_1,\ \cdots,\ k_n)$. For any pair $(l,m)$ of positive integers less than or equal to $n$, we have a ``switch operator'' $S_{(l,m)}: S_n \to S_n$ defined by
\[(i_1,\ \cdots,\ i_n) \mapsto (i_1,\ \cdots,\ i_{l-1},\ i_m,\ i_{m-1},\ \cdots,\ i_l,\ i_{m+1},\ \cdots,\ i_n)\]
The relevant property about this switch operator is
\begin{lemm}\label{switchLemm1}Suppose that $l \leq m$, $x_{h_m} \leq x_{h_l}$, and $y_{k_{l-1}} \leq y_{k_{m}}$. Then $\overline{R}(S_{(l-1,m+1)}(H),S_{(1-1,m)}(K)) \leq \overline{R}(H,K)$
\end{lemm}
\begin{proof}This is a direct calculation
\begin{align*}
\overline{R}(H,K) - \overline{R}(S_{(l,m)}(H),S_{(l,m-1)}(K)) &= (y_{k_{l-1}} - x_{h_l})^2 + (y_{k_{m}}-x_{h_m})^2\\
                                        &\ - (y_{k_{l-1}} - x_{h_m})^2 - (y_{k_{m}}-x_{h_l})^2\\
                                        &= -2y_{k_{l-1}}x_{h_l} - 2y_{k_{m}}x_{h_m}\\
                                        &\ + 2y_{k_{l-1}}x_{h_m} + 2y_{k_{m}}x_{h_l}\\
                                        &= 2(y_{k_{m}} - y_{k_{l-1}})(x_{h_l}-x_{h_m})\\
                                        &\geq 0
\end{align*}
\end{proof}
The roles of $l$, $m$, $H$, and $K$ are all symmetric. We can permute their roles around to get the following three lemmas.
\begin{lemm}\label{switchLemm2}Suppose that $l \leq m$, $x_{h_l} \leq x_{h_m}$, and $y_{k_{m}} \leq y_{k_{l-1}}$. Then $\overline{R}(S_{(l-1,m+1)}(H),S_{(l-1,m)}(K)) \leq \overline{R}(H,K)$
\end{lemm}
\begin{lemm}\label{switchLemm3}Suppose that $l \leq m$, $y_{k_m} \leq y_{k_l}$, and $x_{h_{l}} \leq x_{h_{m+1}}$. Then $\overline{R}(S_{(l,m+1)}(H),S_{(l-1,m+1)}(K)) \leq \overline{R}(H,K)$
\end{lemm}
\begin{lemm}\label{switchLemm4}Suppose that $l \leq m$, $y_{k_l} \leq y_{k_m}$, and $x_{h_{m+1}} \leq x_{h_{l-1}}$. Then $\overline{R}(S_{(l,m+1)}(H),S_{(l-1,m+1)}(K)) \leq \overline{R}(H,K)$
\end{lemm}
The proofs of these statements are all essentially the same. Now we return to problem of minimizing $\overline{R}$. Start with $(H,K) \in S_n\times S_n$. We will keep applying switch operators with the help of the above lemmas to produce a sequence $\left\{(H^{(j)},K^{(j)})\right\}_{j=0}^N$. Set
\[H^{(j)} \equiv (h^{(j)}_1,\ \cdots,\ h^{(j)}_n)\]
and
\[K^{(j)} \equiv (k^{(j)}_1,\ \cdots,\ k^{(j)}_n)\]
The sequence $\left\{H^{(j)},K^{(j)}\right\}$ will have the following properties
\begin{enumerate}
\item $\overline{R}\left(H^{(j+1)},K^{(j+1)}\right) \leq \overline{R}\left(H^{(j)},K^{(j)}\right)$
\item $h^{(1)}_1 = 1$, $h^{(2)}_1 = 1$, $k^{(2)}_n = 1$, $h_1^{(3)} = 1$, $k_n^{(3)} = 1$, $h_n^{(3)} = 2$, $h_1^{(4)} = 1$, $k_n^{(4)} = 1$, $h_n^{(4)} = 2$, $k_1^{(4)} = 2$, etc. That is, each $\left(H^{(j)},K^{(j)}\right)$ agrees with $(I,J)$ on one more index until $\left(H^{(N)},K^{(N)}\right) = (I,J)$.
\end{enumerate}
The construction of this sequence will finish the proof of Lemma \ref{diagLemm}.

The construction of the sequence is inductive. However, writing out the induction formally is a pain since depending on the index, a different one of the above lemmas is required for the inductive step. So we will construct the first few terms of the sequence, and it should then be clear to the reader how to continue. Set $\left(H^{(0)},K^{(0)}\right) \equiv (H,K)$. If $x_{h_1} \leq x_{h_l}$ for all $l$ then we must have $x_{h_1} = x_1$. After a relabeling of $H^{(0)}$ we may take $h_1 = 1$. Now suppose that there exists $l$ such that $x_{h_l} < x_{h_1}$. By assumption we have $y_0 \leq y_{k_l}$. Hence we can apply Lemma \ref{switchLemm1} to produce $(H^{(1)},K^{(1)}) \in S_n \times S_n$ such that $\overline{R}(H^{(1)},K^{(1)}) \leq \overline{R}(H^{(0)},K^{(0)})$ and $h^{(1)}_1 = 1$. In either case we now have $(H^{(1)},K^{(1)})$ such that $h^{(1)}_1 = 1$ and $\overline{R}(H^{(1)},K^{(1)}) \leq \overline{R}(H^{(0)},K^{(0)})$. Next we can apply the same argument using Lemma \ref{switchLemm4} to produce $(H^{(2)},K^{(2)})$ where $h^{(2)}_1 = 1$ and $k^{(2)}_n = 1$. Next, If $x_{h^{(2)}_{n}} = x_2$ then we can relabel $H^{(2)}$ so that $h^{(2)}_n = 2$ and set $(H^{(3)},K^{(3)}) = (H^{(2)},K^{(2)})$. Otherwise, suppose we have some $l\geq 2$ such that $x_{h^{(2)}_l} < x_{h^{(2)}_n}$. By construction of $K^{(2)}$, we have $y_{k^{(2)}_n} \leq y_{k^{(2)}_l}$. Hence we can apply Lemma \ref{switchLemm2} and produce $(H^{(3)},K^{(3)})$ such that $h^{(3)}_n = 2$, $h^{(3)}_1 = 1$, $k^{(3)}_n = 1$, and $\overline{R}(H^{(3)},K^{(3)}) \leq \overline{R}(H^{(2)},K^{(2)})$. The form of the induction should now be clear.
\end{proof}
To show that $D_k(f^{\dagger}) \leq D_k(f)$ we will mimic the corresponding step in the proof of \ref{symmTheo}. That is, we note that the values of $f^{\dagger}$ are just a permutation along the diagonal slices of the values of $f$. The permutation corresponding to $f^{\dagger}$ is exactly the minimizing one of Lemma \ref{diagLemm}. This should be immediately clear once we write out an example: Choose some $k \in \ZZ$ and suppose that $f_1^{(k)} \leq f_2^{(k)} \leq \cdots \leq f_5^{(k)}$ and $f_1^{(k-1)} \leq f_2^{(k-1)} \leq \cdots \leq f_5^{(k-1)}$ be the values of $f$ along the $k$th slice and the $k-1$st slice respectively. Furthermore, suppose that $k$ is odd (otherwise the picture is flipped). Then the values of $f^{\dagger}$ along the $k$ and $k-1$st slice in $G^{\dagger}$ will be
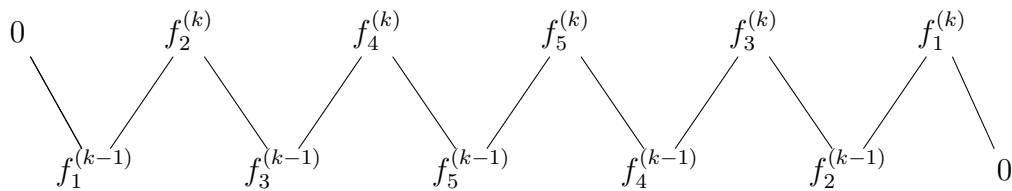
\begin{center}
    \begin{tikzpicture}[description/.style={fill=white,inner sep=2pt}]
                    \matrix (m) [matrix of math nodes, row sep=3em,
                    column sep=.3em, text height=1.5ex, text depth=0.25ex]
                    { 0 &             & f^{(k)}_2&            & f^{(k)}_4&             & f^{(k)}_5 &             & f^{(k)}_3 &              & f^{(k)}_1  &  \\
                        &  f^{(k-1)}_1&          &f^{(k-1)}_3 &          & f^{(k-1)}_5 &           & f^{(k-1)}_4 &           & f^{(k-1)}_2  &   &  0 \\};
                    \path[-,font=\scriptsize]
                    (m-1-1) edge node[auto] {} (m-2-2)
                    (m-2-2) edge node[auto] {} (m-1-3)
                    (m-1-3) edge node[auto] {} (m-2-4)
                    (m-2-4) edge node[auto] {} (m-1-5)
                    (m-1-5) edge node[auto] {} (m-2-6)
                    (m-1-1) edge node[auto] {} (m-2-2)
                    (m-2-8) edge node[auto] {} (m-1-9)
                    (m-1-9) edge node[auto] {} (m-2-10)
                    (m-2-6) edge node[auto] {} (m-1-7)
                    (m-1-7) edge node[auto] {} (m-2-8)
                    (m-2-10) edge node[auto]{} (m-1-11)
                    (m-1-11) edge node[auto] {}(m-2-12);
    \end{tikzpicture}
\end{center}
This concludes the proof of Theorem \ref{diagTheo}.
\section{The Geometry of Minimizing Subgraphs}
\begin{defi}We say that a subgraph $G$ is strongly connected if
\begin{enumerate}
\item $(x,y_1)$ and $(x,y_2)$ in $G$ imply that $(x,y)$ lies in $G$ for all integers $y \in [y_1,y_2]$.
\item $(x_1,y)$ and $(x_2,y)$ in $G$ imply that $(x,y)$ lies in $G$ for all integers $x \in [x_1,x_2]$.
\end{enumerate}
\end{defi}
\begin{prop}\label{stronglyConnected}If $G$ is a minimizing subgraph then $G$ is strongly connected.
\end{prop}
\begin{proof}This follows immediately from Theorem \ref{symmTheoStrict} and the corresponding version for vertical symmetrization.
\end{proof}
\begin{defi}We say that a subgraph is ``walled-in'' if it is strongly connected and
\begin{enumerate}
    \item There exists some $h \in \ZZ$ so that the $h$th horizontal slice walls in every other horizontal slice (see Definition \ref{wall})
    \item There exists some $k \in \ZZ$ so that the $k$th vertical slice walls in every other vertical slices
\end{enumerate}
\end{defi}
\begin{prop}\label{walled-in}If $G$ is a minimizing subgraph then it must be walled-in.
\end{prop}
\begin{proof}From Proposition \ref{stronglyConnected} we know that $G$ is strongly connected. Hence, we just need to verify the ``walling in'' property. We will first show that a horizontal slice exists which walls in all other horizontal slices. Let $k$ be the largest integer so that $G^k$, the $k$th horizontal slice of $G$, is non-empty. Then, by Theorem \ref{symmTheoWall}, either $G^k$ walls in $G^{k-1}$ or $G^{k-1}$ walls in $G^k$. Moving down the graph in this fashion, we can find an integer $m$ (possibly equal to $0$) such that $G^{k-m}$ walls in $G^{k-i}$ for all $i = 0$, $1$, $\cdots$, $m-1$, and $m+1$, i.e. $G^{k-m}$ walls in all of the slices above it and the slice immediately below.
In the graph below we have highlighted the $G^{k-m}$th slice.
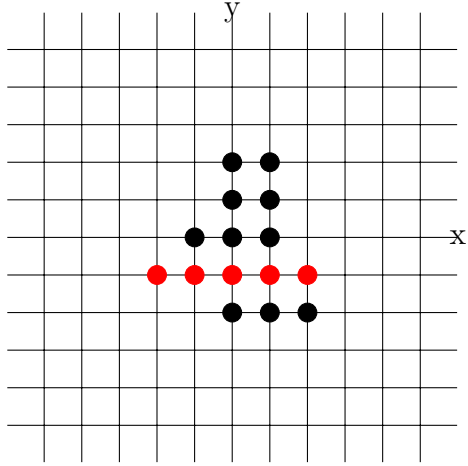
\begin{center}
    \begin{tikzpicture}
        [interior/.style={circle,draw=black,fill=black, inner sep=0pt,minimum size = 2.5mm},
        boundary/.style={circle,draw=black,fill=black!60, inner sep=0pt,minimum size = 2.5mm},
        exterior/.style={circle,draw=black,fill=white, inner sep=0pt,minimum size = 2.5mm},
        highlight/.style = {circle,draw= red, fill = red, inner sep=0pt, minimum size = 2.5mm}]
        \draw[step=.5cm] (-2.99,-2.99) grid (2.99,2.99);
        \draw (0,3) node {y};
        \draw (3,0) node {x};
        \draw (0,0) node [interior]{};
        \draw (.5,0) node [interior]{};
        \draw (-.5,0) node [interior]{};
        \draw (0,.5) node [interior]{};
        \draw (.5,.5) node [interior]{};
        \draw (0,1) node [interior]{};
        \draw (.5,1) node [interior]{};
        \draw (0,-.5) node [highlight]{};
        \draw (.5,-.5) node [highlight]{};
        \draw (1,-.5) node [highlight]{};
        \draw (-.5,-.5) node [highlight]{};
        \draw (-1,-.5) node [highlight]{};
        \draw (0,-1) node [interior]{};
        \draw (.5,-1) node [interior]{};
        \draw (1,-1) node [interior]{};
    \end{tikzpicture}
\end{center}
We will show that $G^{k-m}$ must wall in all horizontal slices. For the sake of contradiction, suppose that there exists some integer $l > m$ so that $G^{k-m}$ does not wall in $G^{k-l}$. Furthermore, let $l$ be the smallest such integer. In the following subgraph, $G^{k-l}$ could be the bottom slice.
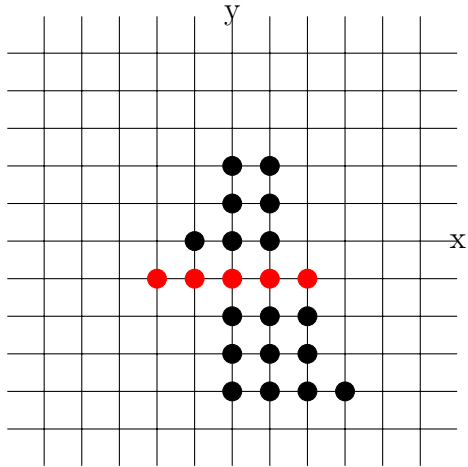
\begin{center}
    \begin{tikzpicture}
        [interior/.style={circle,draw=black,fill=black, inner sep=0pt,minimum size = 2.5mm},
        boundary/.style={circle,draw=black,fill=black!60, inner sep=0pt,minimum size = 2.5mm},
        exterior/.style={circle,draw=black,fill=white, inner sep=0pt,minimum size = 2.5mm},
        highlight/.style = {circle,draw= red, fill = red, inner sep=0pt, minimum size = 2.5mm}]
        \draw[step=.5cm] (-2.99,-2.99) grid (2.99,2.99);
        \draw (0,3) node {y};
        \draw (3,0) node {x};
        \draw (0,0) node [interior]{};
        \draw (.5,0) node [interior]{};
        \draw (-.5,0) node [interior]{};
        \draw (0,.5) node [interior]{};
        \draw (.5,.5) node [interior]{};
        \draw (0,1) node [interior]{};
        \draw (.5,1) node [interior]{};
        \draw (0,-.5) node [highlight]{};
        \draw (.5,-.5) node [highlight]{};
        \draw (1,-.5) node [highlight]{};
        \draw (-.5,-.5) node [highlight]{};
        \draw (-1,-.5) node [highlight]{};
        \draw (0,-1) node [interior]{};
        \draw (.5,-1) node [interior]{};
        \draw (1,-1) node [interior]{};
        \draw (0,-1.5) node [interior]{};
        \draw (.5,-1.5) node [interior]{};
        \draw (1,-1.5) node [interior]{};
        \draw (0,-2) node [interior]{};
        \draw (.5,-2) node [interior]{};
        \draw (1,-2) node [interior]{};
        \draw (1.5,-2) node [interior]{};
    \end{tikzpicture}
\end{center}
Consider the positive horizontal symmetrization of the above subgraph.
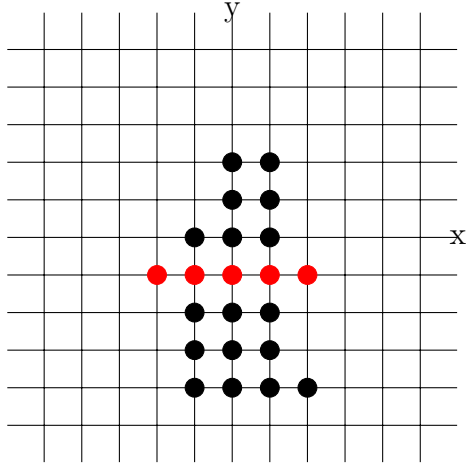
\begin{center}
    \begin{tikzpicture}
        [interior/.style={circle,draw=black,fill=black, inner sep=0pt,minimum size = 2.5mm},
        boundary/.style={circle,draw=black,fill=black!60, inner sep=0pt,minimum size = 2.5mm},
        exterior/.style={circle,draw=black,fill=white, inner sep=0pt,minimum size = 2.5mm},
        highlight/.style = {circle,draw= red, fill = red, inner sep=0pt, minimum size = 2.5mm}]
        \draw[step=.5cm] (-2.99,-2.99) grid (2.99,2.99);
        \draw (0,3) node {y};
        \draw (3,0) node {x};
        \draw (0,0) node [interior]{};
        \draw (.5,0) node [interior]{};
        \draw (-.5,0) node [interior]{};
        \draw (0,.5) node [interior]{};
        \draw (.5,.5) node [interior]{};
        \draw (0,1) node [interior]{};
        \draw (.5,1) node [interior]{};
        \draw (0,-.5) node [highlight]{};
        \draw (.5,-.5) node [highlight]{};
        \draw (1,-.5) node [highlight]{};
        \draw (-.5,-.5) node [highlight]{};
        \draw (-1,-.5) node [highlight]{};
        \draw (0,-1) node [interior]{};
        \draw (.5,-1) node [interior]{};
        \draw (-.5,-1) node [interior]{};
        \draw (0,-1.5) node [interior]{};
        \draw (.5,-1.5) node [interior]{};
        \draw (-.5,-1.5) node [interior]{};
        \draw (0,-2) node [interior]{};
        \draw (.5,-2) node [interior]{};
        \draw (1,-2) node [interior]{};
        \draw (-.5,-2) node [interior]{};
    \end{tikzpicture}
\end{center}
Since symmetrization can only lower $\lam_D$, we must still have a minimizing subgraph. However, the rightmost vertical slice is not connected. This contradicts the vertical version of Theorem \ref{symmTheoStrict}. In fact, this argument works in complete generality. Since all slices in between $G^{k-m}$ and $G^{k-l}$ do not wall in $G^{k-m}$ ($l$ was chosen to be minimal), $(G^{\star})^{k-m}$ will extend farther to the right then all slices in between $(G^{\star})^{k-m}$ and $(G^{\star})^{k-l}$. The same statement holds for $(G^{\star})^{k-l}$. This implies the existence of a non-connected vertical slice in $G^{\star}$ which cannot happen if $G^{\star}$ is a minimizing subgraph. Hence, no such $l$ exists.

To prove the same statement for vertical slices we simply note that rotating $\ZZ^2$ by 90 degrees is an automorphism of $\ZZ^2$.
\end{proof}
\begin{coro}If $G$ is a minimizing subgraph, then $\mathbf{G}$ is simply connected.
\end{coro}
\begin{proof} Let $l$ be the horizontal line segment which walls in $\mathbf{G}$. Since all vertical paths are continuous, every point $x$ in $\mathbf{G}$ lies on a vertical line segment $s_x$ entirely contained in $\mathbf{G}$, which starts at $x$ and ends at $l$. This is easily seen to imply that $G$ is contractible
\end{proof}
Now we show that minimizing subgraphs cannot be too thin.
\begin{prop}\label{diamBound}For $\Omega \subset \RR^2$, let $D(\Omega)$ denote the diameter of $\Omega$. Then there exists $C > 0$ such that for any minimizing subgraph $G$
\[D(\mathbf{G}) \leq C\sqrt{|G|}\]
\end{prop}
\begin{proof}Set $n = |G|$. Let $W$ and $H$ be the length of the longest horizontal slice and longest vertical slice respectively. Then Proposition \ref{walled-in} implies that $\mathbf{G}$ is contained inside a $W+1/2$ by $H+1/2$ rectangle. Then
\[D(\mathbf{G}) \leq \sqrt{(W+1/2)^2 + (H+1/2)^2}\]
Thus, it suffices to prove that $H$ and $W$ are both $O(\sqrt{n})$. Clearly it suffices to only prove that $H = O(\sqrt{n})$.

We have a vertical slice of length $H$ in $G$. Below we draw such a slice for $H = 11$.
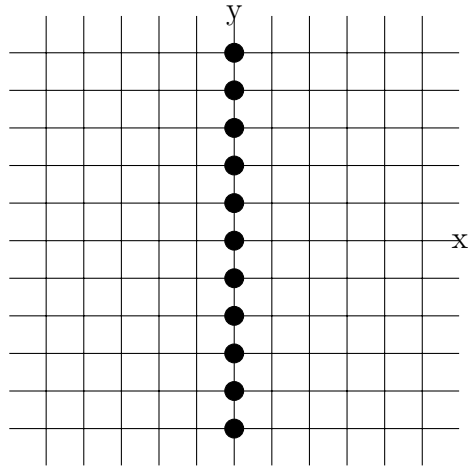
\begin{center}
\begin{tikzpicture}
        [interior/.style={circle,draw=black,fill=black, inner sep=0pt,minimum size = 2.5mm},
        boundary/.style={circle,draw=black,fill=black!60, inner sep=0pt,minimum size = 2.5mm},
        exterior/.style={circle,draw=black,fill=white, inner sep=0pt,minimum size = 2.5mm},
        highlight/.style = {circle,draw= red, fill = red, inner sep=0pt, minimum size = 2.5mm}]
        \draw[step=.5cm] (-2.99,-2.99) grid (2.99,2.99);
        \draw (0,3) node {y};
        \draw (3,0) node {x};
        \draw (0,2.5) node [interior]{};
        \draw (0,2) node [interior]{};
        \draw (0,1.5) node [interior]{};
        \draw (0,1) node [interior]{};
        \draw (0,.5) node [interior]{};
        \draw (0,0) node [interior]{};
        \draw (0,-.5) node [interior]{};
        \draw (0,-1) node [interior]{};
        \draw (0,-1.5) node [interior]{};
        \draw (0,-2) node [interior]{};
        \draw (0,-2.5) node [interior]{};
    \end{tikzpicture}
\end{center}
Now we consider $G^{\dagger}$, the diagonal symmetrization of $G$. From Theorem \ref{diagTheo}, $G^{\dagger}$ is still a minimizing subgraph. Due to the presence of the slice of length $H$ in $G$, we can find a ``diagonal path'' of length $H$ in $G^{\dagger}$.
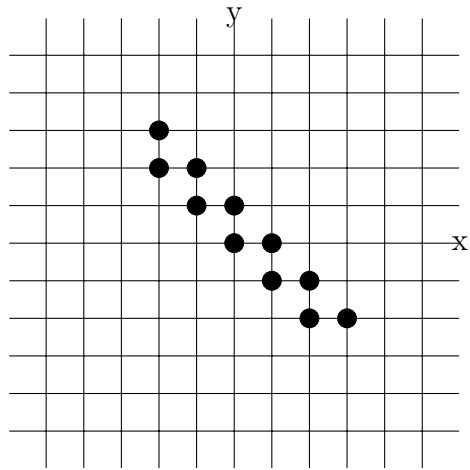
\begin{center}
\begin{tikzpicture}
        [interior/.style={circle,draw=black,fill=black, inner sep=0pt,minimum size = 2.5mm},
        boundary/.style={circle,draw=black,fill=black!60, inner sep=0pt,minimum size = 2.5mm},
        exterior/.style={circle,draw=black,fill=white, inner sep=0pt,minimum size = 2.5mm},
        highlight/.style = {circle,draw= red, fill = red, inner sep=0pt, minimum size = 2.5mm}]
        \draw[step=.5cm] (-2.99,-2.99) grid (2.99,2.99);
        \draw (0,3) node {y};
        \draw (3,0) node {x};
        \draw (0,0) node [interior]{};
        \draw (0,.5) node [interior]{};
        \draw (-.5,.5) node [interior]{};
        \draw (-.5,1) node [interior]{};
        \draw (-1,1) node [interior]{};
        \draw (-1,1.5) node [interior]{};
        \draw (.5,0) node [interior]{};
        \draw (.5,-.5) node [interior]{};
        \draw (1,-.5) node [interior]{};
        \draw (1,-1) node [interior]{};
        \draw (1.5,-1) node [interior]{};
    \end{tikzpicture}
\end{center}
From Proposition \ref{walled-in}, $G^{\dagger}$ must be walled-in. Hence, we have a horizontal slice and a vertical slice that both ``wall in'' this diagonal slice.
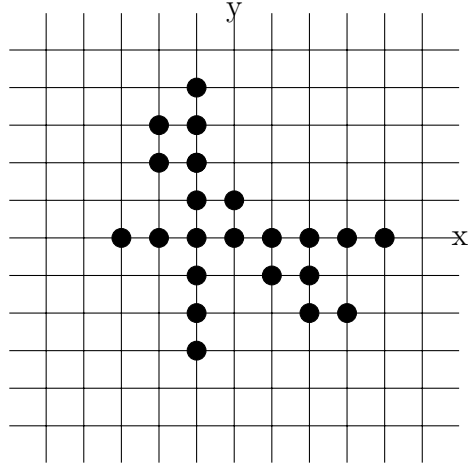
\begin{center}
\begin{tikzpicture}
        [interior/.style={circle,draw=black,fill=black, inner sep=0pt,minimum size = 2.5mm},
        boundary/.style={circle,draw=black,fill=black!60, inner sep=0pt,minimum size = 2.5mm},
        exterior/.style={circle,draw=black,fill=white, inner sep=0pt,minimum size = 2.5mm},
        highlight/.style = {circle,draw= red, fill = red, inner sep=0pt, minimum size = 2.5mm}]
        \draw[step=.5cm] (-2.99,-2.99) grid (2.99,2.99);
        \draw (0,3) node {y};
        \draw (3,0) node {x};
        \draw (0,0) node [interior]{};
        \draw (0,.5) node [interior]{};
        \draw (-.5,.5) node [interior]{};
        \draw (-.5,1) node [interior]{};
        \draw (-1,1) node [interior]{};
        \draw (-1,1.5) node [interior]{};
        \draw (.5,0) node [interior]{};
        \draw (.5,-.5) node [interior]{};
        \draw (1,-.5) node [interior]{};
        \draw (1,-1) node [interior]{};
        \draw (1.5,-1) node [interior]{};
        \draw (-.5,1) node [interior]{};
        \draw (-.5,1.5) node [interior]{};
        \draw (-.5,2) node [interior]{};
        \draw (-.5,-.5) node [interior]{};
        \draw (-.5,-1) node [interior]{};
        \draw (-.5,-1.5) node [interior]{};
        \draw (1,0) node [interior]{};
        \draw (1.5,0) node [interior]{};
        \draw (2,0) node [interior]{};
        \draw (-.5,0) node [interior]{};
        \draw (-1,0) node [interior]{};
        \draw (-1.5,0) node [interior]{};
    \end{tikzpicture}
\end{center}
At least $H - 2$ points on the diagonal path do not lie on these horizontal and vertical slices. These $H-2$ points lie inside the rectangle determined by the horizontal and vertical slices. Furthermore, these horizontal and vertical slices divide the rectangle into four quadrants. Thus, $(H-2)/4$ points must lie in at least one of these quadrants. Now we focus our attention on this quadrant. A representative picture might look like
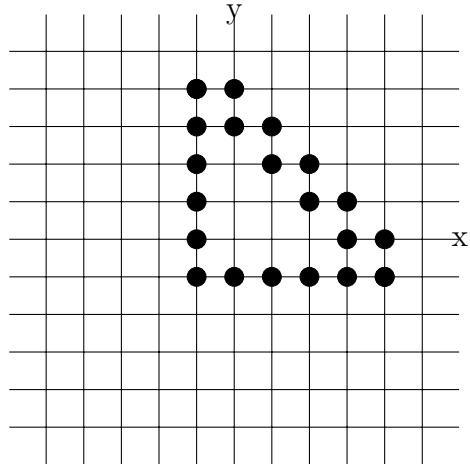
\begin{center}
\begin{tikzpicture}
        [interior/.style={circle,draw=black,fill=black, inner sep=0pt,minimum size = 2.5mm},
        boundary/.style={circle,draw=black,fill=black!60, inner sep=0pt,minimum size = 2.5mm},
        exterior/.style={circle,draw=black,fill=white, inner sep=0pt,minimum size = 2.5mm},
        highlight/.style = {circle,draw= red, fill = red, inner sep=0pt, minimum size = 2.5mm}]
        \draw[step=.5cm] (-2.99,-2.99) grid (2.99,2.99);
        \draw (0,3) node {y};
        \draw (3,0) node {x};
        \draw (-.5,-.5) node [interior]{};
        \draw (0,-.5) node [interior]{};
        \draw (.5,-.5) node [interior]{};
        \draw (1,-.5) node [interior]{};
        \draw (1.5,-.5) node [interior]{};
        \draw (2,-.5) node [interior]{};
        \draw (-.5,0) node [interior]{};
        \draw (-.5,.5) node [interior]{};
        \draw (-.5,1) node [interior]{};
        \draw (-.5,1.5) node [interior]{};
        \draw (-.5,2) node [interior]{};
        \draw (2,-.5) node [interior]{};
        \draw (2,0) node [interior]{};
        \draw (1.5,0) node [interior]{};
        \draw (1.5,.5) node [interior]{};
        \draw (1,.5) node [interior]{};
        \draw (1,1) node [interior]{};
        \draw (.5,1) node [interior]{};
        \draw (.5,1.5) node [interior]{};
        \draw (0,1.5) node [interior]{};
        \draw (0,2) node [interior]{};
        \draw (-.5,2) node [interior]{};
    \end{tikzpicture}
\end{center}
Proposition \ref{stronglyConnected} implies that all horizontal and vertical slices must be connected. Hence, the highlighted points must also be in the graph.
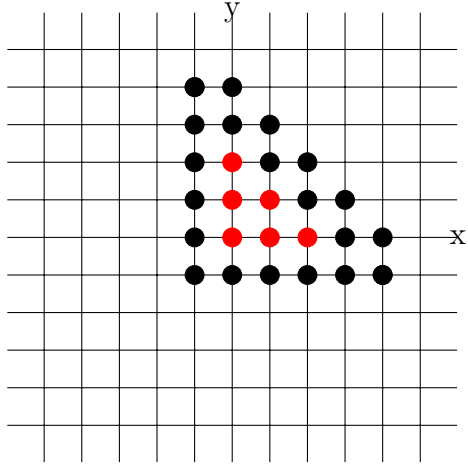
\begin{center}
\begin{tikzpicture}
        [interior/.style={circle,draw=black,fill=black, inner sep=0pt,minimum size = 2.5mm},
        boundary/.style={circle,draw=black,fill=black!60, inner sep=0pt,minimum size = 2.5mm},
        exterior/.style={circle,draw=black,fill=white, inner sep=0pt,minimum size = 2.5mm},
        highlight/.style = {circle,draw= red, fill = red, inner sep=0pt, minimum size = 2.5mm}]
        \draw[step=.5cm] (-2.99,-2.99) grid (2.99,2.99);
        \draw (0,3) node {y};
        \draw (3,0) node {x};
        \draw (-.5,-.5) node [interior]{};
        \draw (0,-.5) node [interior]{};
        \draw (.5,-.5) node [interior]{};
        \draw (1,-.5) node [interior]{};
        \draw (1.5,-.5) node [interior]{};
        \draw (2,-.5) node [interior]{};
        \draw (-.5,0) node [interior]{};
        \draw (-.5,.5) node [interior]{};
        \draw (-.5,1) node [interior]{};
        \draw (-.5,1.5) node [interior]{};
        \draw (-.5,2) node [interior]{};
        \draw (2,-.5) node [interior]{};
        \draw (2,0) node [interior]{};
        \draw (1.5,0) node [interior]{};
        \draw (1.5,.5) node [interior]{};
        \draw (1,.5) node [interior]{};
        \draw (1,1) node [interior]{};
        \draw (.5,1) node [interior]{};
        \draw (.5,1.5) node [interior]{};
        \draw (0,1.5) node [interior]{};
        \draw (0,2) node [interior]{};
        \draw (-.5,2) node [interior]{};
        \draw (0,0) node [highlight]{};
        \draw (.5,0) node [highlight]{};
        \draw (1,0) node [highlight]{};
        \draw (0,.5) node [highlight]{};
        \draw (.5,.5) node [highlight]{};
        \draw (0,1) node [highlight]{};
    \end{tikzpicture}
\end{center}
In general we can conclude that there are at least
\[\sum_{j=1}^{(H-2)/8}j = (1/2)\left(\frac{H-2}{8}\right)\left(\frac{H-2}{8} + 1\right) = \frac{H^2 + 4H - 12}{128}\]
points in $G$. That is,
\[\frac{H^2 + 4H - 12}{128} \leq n \Rightarrow \]
\[H \leq \sqrt{128n + 12}\]
\end{proof}

\section{Approximation By Continuous Eigenvalues}\label{analysis}
Now we establish some asymptotic estimates for $\lam_D$. This will be accomplished by relating $\lam_D(G)$ to the regular Laplacian eigenvalues of a related domain in $\RR^2$. What follows is a minor modification of ideas used in finite difference approximations to PDEs. See \cite{n7}, \cite{n3}, and \cite{n4}. The goal of this section is to prove Theorem \ref{approxTheo} which we quote here again.
For any bounded domain $\Omega \subset \RR^2$ and $\ep > 0$, we defined $B^{\ell_1}_{\ep}\left(\Omega\right)$ to be the interior of the set of all points with $\ell_1$ distance less than $\ep$ to $\overline{\Omega}$. We also set $\lam(\Omega)$ to be the first Dirichlet eigenvalue of the regular Laplacian. For subgraphs $G$ with $n$ vertices we will prove
    \begin{theo}For some constant $C > 0$
    \[\frac{\lam\left(B^{\ell_1}_{2/\sqrt{n}}\left(\mathbf{G}^*\right)\right)}{n + C\lam\left(B^{\ell_1}_{2/\sqrt{n}}\left(\mathbf{G}^*\right)\right)} \leq \lam_D(G) \leq \frac{\lam\left(\mathbf{G}^*\right)}{n - C\lam\left(\mathbf{G}^*\right)}\]
    \end{theo}
\subsection{An Upperbound for $\lam_D$}
Here we prove
\begin{theo}\label{upperBound} Let $n = |G|$. Then
$$\lam_D(G) \leq \frac{\pi^2\lam(\mathbf{G}^*)}{\pi^2n - \lam(\mathbf{G}^*)}$$
\end{theo}
\begin{proof}An extremely close variant of this is proved in \cite{n3}. We will adapt the ideas there to the case at hand. Recall that $\overline{G}$ was the union $G$ and $\pa G$. The operator $L_D$ acted on functions defined on $\overline{G}$ which vanished on $\pa G$. Now embed $\overline{G}$ and $G$ into $\RR^2$ by sending $(x,y)$ to $((1/\sqrt{n})x,(1/\sqrt{n})y)$. Denote these embeddings by $\overline{G}^*$ and $G^*$ respectively. Let $C_g$ be a closed square of volume $1/n$ centered at $g$. Observe that $\mathbf{G}^*$ is the interior of $\bigcup_{g \in G^*}C_g$. Let $u$ be an eigenfunction associated to $\lam(\mathbf{G}^*)$ with $\vv u\vv_2 = 1$. Extend $u$ to be $0$ outside of $\mathbf{G}^*$. To prove Theorem \ref{upperBound} we will create a ``discrete version'' of $u$ on $\overline{G}^*$ and then plug it into the relevant Rayleigh quotient.

Define a function $v$ on $\overline{G}^*$ by averaging values of $u$, i.e.
\[v(g) = n\int_{C_g}u(x,y)\ dxdy\]
From the Rayleigh quotient we have
$$\lam_D(G) = \lam_D(G^*) \leq \frac{\sum_{i \sim j}(v(i) - v(j))^2}{\sum_{i}v^2(i)}$$
We will proceed by bounding the numerator and denominator of this. We start with the following simple calculation
\begin{align*}
\sum_{g\in \overline{G}^*}\int_{C_g}\left[u(x,y)-v(g)\right]^2\ dxdy &= \sum_{g\in \overline{G}^*}\Big[\int_{C_g}u^2(x,y)\ dxdy\\
                                                          &\ - 2v(g)\int_{C_g}u(x,y)\ dxdy + v^2(g)\int_{C_g}\ dxdy\Big]\\
                                                          &= \int_{\overline{G}^*}u^2(x,y)\ dxdy + \sum_{g \in \overline{G}^*}\Big[\frac{-2}{n}v^2(g) + \frac{1}{n}v^2(g)\Big]\\
                                                          &= \int_{\overline{G}^*}u^2(x,y)\ dxdy - \frac{1}{n}\sum_{g \in \overline{G}^*}v^2(g)
\end{align*}
As a special case of the main Theorem in \cite{n10}, we have the following version of the Poincare Inequality:
\begin{theo}\label{poincare}Suppose $S \subset \RR^2$ is a square, $f \in H^1(S)$, and $\int_{S} f(x,y)\ dxdy = 0$. Then
\[\int_{S}|f(x,y)|^2\ dxdy \leq \frac{l^2}{\pi^2}\int_{S}|\nabla f(x,y)|^2\ dxdy\]
where $l$ is the side length of $S$.
\end{theo}
By construction
\begin{align*}
\int_{C_g}[u(x,y)-v(g)]\ dxdy &= \int_{C_g}u(x,y)\ dxdy - \int_{C_g}n\left(\int_{C_g}u(x,y)\ dxdy\right)\\
                              &=\int_{C_g}u(x,y)\ dxdy - \int_{C_g}u(x,y)\ dxdy\left(\int_{C_g}n\ dxdy\right)\\
                              &= 0
\end{align*}
Hence, we can apply \ref{poincare} to get
\begin{align*}
\int_{\mathbf{G}^*}u^2(x,y)\ dxdy - \frac{1}{n}\sum_{g\in \overline{G}^*} v^2(g) &= \sum_{g\in \overline{G}^*}\int_{C_g}[u(x,y)-v(g)]^2\ dxdy\\
                                                                           &\leq \sum_{g\in \overline{G}^*}\frac{1}{n\pi^2}\int_{C_g}|\nabla u|^2\ dxdy\\
                                                                           &=\frac{1}{n\pi^2}\int_{\mathbf{G}^*}|\nabla u|^2\ dxdy\\
                                                                           &= \frac{\lam(\mathbf{G}^*)}{n\pi^2}
\end{align*}
This implies that
\begin{align*}
\int_{\mathbf{G}^*}u^2(x,y)\ dxdy - \frac{1}{n}\sum_{g\in\overline{G}^*}v^2(g) &\leq \frac{\lam(\mathbf{G}^*)}{n\pi^2} \Leftrightarrow\\
\sum_{g \in \overline{G}^*}v^2(g) &\geq n\int_{\mathbf{G}^*}u^2(x,y)\ dxdy - \frac{\lam(\mathbf{G}^*)}{\pi^2} \Leftrightarrow\\
\sum_{g \in \overline{G}^*}v^2(g) &\geq n - \frac{\lam(\mathbf{G}^*)}{\pi^2}
\end{align*}
The final inequality follows because $u$ was chosen to be a normalized eigenfunction.
This will provide the denominator bound.

Now for the numerator: For $(a,b) \in \overline{G}^*$, set $v_{(a,b)}(x,y) = v(a+x,b+y)$ and $u_{(a,b)}(x,y) = u(a+x,b+y)$. Also, we set $h = 1/\sqrt{n}$. Now fix some $g = (a,b) \in \overline{G}^*$. We need to control $v_g(h,0) - v_g(0,0)$ in terms of $u$. We will use integration by parts to rewrite $v_{g}(h,0) - v_{g}(0,0)$ as an integral of a bump function times a partial derivative of $u$. This will allow us to relate $\sum (v(a,b)-v(c,d))^2$ to $\int_{\mathbf{G}^*}|\nabla u|^2$. Define a bump function by
\begin{equation*}
    \psi(x) = \left\{
    \begin{array}{rl}
        x + \frac{h}{2} & \text{if } x \in [-\frac{h}{2},\frac{h}{2}]\\
        \frac{3h}{2} - x & \text{if } x \in [\frac{h}{2},\frac{3h}{2}]\\
        0 & \text{otherwise}
    \end{array} \right.
\end{equation*}
Extend $v$ to be zero on any mess points outside of $\overline{G}^*$. Recall that earlier we extended $u$ to be $0$ outside of its original domain. Then
\begin{lemm}\label{annoyingLemm}
\[\int_{\RR^2}|\nabla u|^2\ dxdy - \sum_{i \sim_{h\ZZ^2} j}(v(i)-v(j))^2 = \]
\[\frac{1}{h}\sum_{g\in h\ZZ^2}\int_{-\frac{h}{2}}^{\frac{3h}{2}}\int_{-\frac{h}{2}}^{\frac{h}{2}}\psi(x)\left(\frac{\pa u_g}{\pa x}(x,y) - \frac{1}{h}\left(v_g(h,0) - v_g(0,0)\right)\right)^2\ dxdy + \]
\[\frac{1}{h}\sum_{g\in h\ZZ^2}\int_{-\frac{h}{2}}^{\frac{3h}{2}}\int_{-\frac{h}{2}}^{\frac{h}{2}}\psi(y)\left(\frac{\pa u_g}{\pa y}(x,y) - \frac{1}{h}\left(v_g(0,h) - v_g(0,0)\right)\right)^2\ dydx\]
\end{lemm}
\begin{proof}This is a quite messy computation which we relegate to the appendix. We strongly encourage the first (or second) time reader to skip this proof.
\end{proof}
This implies
\[\sum_{i \sim_{h\ZZ^2} j}(v(i)-v(j))^2 \leq \int_{\RR^2}|\nabla u|^2\ dxdy\]
This successfully bounds the numerator of the Rayleigh quotient and we now have
\[\lam_D(G) = \lam_D(G^*) \leq \frac{\sum_{(a,b)\sim (c,d)}(v(a,b) - v(c,d))^2}{\sum_{(a,b)}v^2(a,b)} \leq \frac{\lam(\mathbf{G}^*)}{n-\frac{\lam(\mathbf{G}^*)}{\pi^2}} = \frac{\pi^2\lam(\mathbf{G}^*)}{\pi^2n-\lam(\mathbf{G}^*)}\]
\end{proof}
\subsection{A Lower Bound for $\lam_D$}
Here we will again make quite minor modifications to ideas presented in \cite{n7}, \cite{n3}, and \cite{n4}. As in the previous section we consider $\overline{G}^*$ and the domain $\mathbf{G}^*$. Recall that for a bounded open domain $\Omega \subset \RR^2$ we defined $B^{\ell_1}_{\ep}(\Omega)$ to be the interior of the set of all points with $\ell_1$ distance less than $\ep$ to $\Omega$. We have $\overline{G}^* \subset B^{\ell_1}_{2/\sqrt{n}}\left(\mathbf{G}^*\right)$.
Our strategy in this section is similar to the previous one. We will start with a principle eigenfunction for the graph and produce an approximate eigenfunction for $B^{\ell_1}_{2/\sqrt{n}}\left(\mathbf{G}^*\right)$. Then plugging everything into the relevant Rayleigh quotient will give us
\begin{prop}
\[\lam_D(G) \geq \frac{\lam\left(B^{\ell_1}_{2/\sqrt{n}}\left(\mathbf{G}^*\right)\right)}{n+(5/12)\lam\left(B^{\ell_1}_{2/\sqrt{n}}\left(\mathbf{G}^*\right)\right)}\]
\end{prop}
\begin{proof}
We start by adding edges to the mesh by drawing lines of slope $-1$ through every mesh point as in the following picture
    
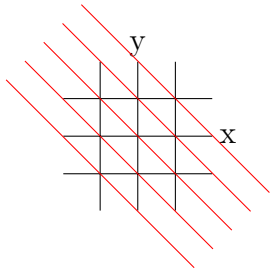
\begin{center}
\begin{tikzpicture}
[interior/.style={circle,draw=black,fill=black, inner sep=0pt,minimum size = 2.5mm},
 boundary/.style={circle,draw=black,fill=black!60, inner sep=0pt,minimum size = 2.5mm},
 exterior/.style={circle,draw=black,fill=white, inner sep=0pt,minimum size = 2.5mm},
 highlight/.style = {circle,draw= red, fill = red, inner sep=0pt, minimum size = 2.5mm}]
\draw[step=.5cm] (-.99,-.99) grid (.99,.99);
\draw (0,1.2) node {y};
\draw (1.2,0) node {x};
\draw[red] (-1.75,.75) -- (.75,-1.75);
\draw[red] (-1.5,1) -- (1,-1.5);
\draw[red] (-1.25,1.25) -- (1.25,-1.25);
\draw[red] (-1,1.5) -- (1.5,-1);
\draw[red] (-.75,1.75) -- (1.75,-.75);
\end{tikzpicture}
\end{center}
If we want to include these added edges we will refer to the subgraph as $G_{T}^*$ where the $T$ stands for triangle.
Let $u$ be a normalized eigenfunction for $G^*$. Extend $u$ to be $0$ everywhere it is not defined. We will define a continuous and piecewise differentiable function $v$ vanishing on the boundary of $B^{\ell_1}_{2/\sqrt{n}}\left(\mathbf{G}^*\right)$. Choose a triangle in the mesh oriented like
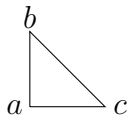
\begin{center}
\begin{tikzpicture}
[interior/.style={circle,draw=black,fill=black, inner sep=0pt,minimum size = 2.5mm},
 boundary/.style={circle,draw=black,fill=black!60, inner sep=0pt,minimum size = 2.5mm},
 exterior/.style={circle,draw=black,fill=white, inner sep=0pt,minimum size = 2.5mm},
 highlight/.style = {circle,draw= red, fill = red, inner sep=0pt, minimum size = 2.5mm}]
\draw (0,0) -- (0,1) -- (1,0) -- (0,0);
\draw (-.2,0) node {$a$};
\draw (0,1.2) node {$b$};
\draw (1.2,0) node {$c$};
\end{tikzpicture}
\end{center}
Let $(s,t)$ be $x$ and $y$ coordinates on the triangle. Then define
\[v(s,t) = \sqrt{n}(u(c)-u(a))s + \sqrt{n}(u(b)-u(a))t + u(a)\]
On the other triangles, define $v$ in the obvious way. It is clear that $v$ is continuous and piecewise differentiable. Also, we note that $v$ vanishes on the boundary of $B^{\ell_1}_{2/\sqrt{n}}\left(\mathbf{G}^*\right)$. Thus we have
\[\lam\left(B^{\ell_1}_{2/\sqrt{n}}\left(\mathbf{G}^*\right)\right) \leq \frac{\int_{B^{\ell_1}_{2/\sqrt{n}}\left(\mathbf{G}^*\right)} |\nabla v|^2}{\int_{B^{\ell_1}_{2/\sqrt{n}}\left(\mathbf{G}^*\right)} v^2}\]
It is immediately clear that
\[\int_{B^{\ell_1}_{2/\sqrt{n}}\left(\mathbf{G}^*\right)} |\nabla v|^2 = \sum_{i\sim_{\overline{G}^*} j} (u(i)-u(j))^2 = \lam_D(G)\]
The denominator bound is a little more tricky:
\[\int_{B^{\ell_1}_{2/\sqrt{n}}\left(\mathbf{G}^*\right)} v^2 = \]
\[\sum_{\text{triangles } abc}\int_0^{1/\sqrt{n}}\int_0^{1/\sqrt{n} - t}(\sqrt{n}(u(c)-u(a))s + \sqrt{n}(u(b)-u(a))t + u(a))^2\ dsdt =\]
\[\sum_{\text{triangles } abc} \frac{u(a)^2+u(b)^2+u(c)^2 + u(a)u(b) + u(a)u(c)}{12n}\]
Let $h = 1/\sqrt{n}$ and extend $v$ to be $0$ everywhere it is not defined. After noting that each edge lies in two triangles and each vertex lies in six triangles, we see that the above equality gives
\[\int_{B^{\ell_1}_{2/\sqrt{n}}\left(\mathbf{G}^*\right)} v^2 = \]
\[\sum_{(a,b)} \frac{u(a,b)}{12n}(6u(a,b) + u(a,b+h) + u(a-h,b+h)\]
\[ + u(a-h,b) + u(a,b-h) + u(a+h,b-h) + u(a+h,b)) = \]
\[\sum_{(a,b)} u^2(a,b)/n + \frac{u(a,b)}{12n}\Big((u(a,b+h)-u(a,b)) + (u(a-h,b+h) - u(a,b))\]
\[ + (u(a-h,b)-u(a,b)) + (u(a,b-h)-u(a,b))\]
\[ + (u(a,+h,b-h)-u(a,b)) + (u(a+h,b)-u(a,b))\Big) = \]
\[1/n - \frac{1}{12n}\sum_{i\sim_{\overline{G}^*_T} j}(u(i)-u(j))^2 = \]
\[1/n - \frac{\lam_D(G)}{12n} - \frac{1}{12n}\sum_{i\sim_{\overline{G}^*_T} j\text{ and } i \not\sim_{\overline{G}^*} j}(u(i)-u(j))^2\]

The last term involves differences of the eigenfunction evaluated on opposite vertices of the lattice squares. Consider the triangle in the diagram above with vertices $a$,$b$, and $c$. We need to control $-(u(b) - u(c))^2$. We have
\begin{align*}
-(u(b)-u(c))^2 &= -([u(b) - u(a)] + [u(a) - u(c)])^2\\
               &= -(u(b)-u(a))^2 - (u(a)-u(c))^2 - 2[u(b)-u(a)][u(a)-u(c)]\\
               &\geq -2(u(b)-u(a))^2 - 2(u(a)-u(c))^2
\end{align*}
where we used the AM-GM inequality\footnote{$xy \leq \frac{x^2+y^2}{2}$} in the last step.
This implies that
\[1/n - \frac{\lam_D(G)}{12n} - \frac{1}{12n}\sum_{i\sim \overline{G}^*_T j\text{ and } i \not\sim_{\overline{G}^*} j}(u(i)-u(j))^2 \geq\]
\[1/n - \frac{5}{12n}\lam_D(G)\]
Thus we have bounded the denominator of the Rayleigh quotient
\[\int_{B^{\ell_1}_{2/\sqrt{n}}\left(\mathbf{G}^*\right)}v^2 \geq 1/n - \frac{5}{12n}\lam_D(G)\]
Plugging everything into the relevant Rayleigh quotient now gives
\[\lam\left(B^{\ell_1}_{2/\sqrt{n}}\left(\mathbf{G}^*\right)\right) \leq \frac{\lam_D(G)}{1/n - (5/12n)\lam_D(G)}\]
Rearranging this gives
\[\lam_D(G) \geq \frac{\lam\left(B^{\ell_1}_{2/\sqrt{n}}\left(\mathbf{G}^*\right)\right)}{n+(5/12)\lam\left(B^{\ell_1}_{2/\sqrt{n}}\left(\mathbf{G}^*\right)\right)}\]
\end{proof}
This concludes the proof of Theorem \ref{approxTheo}.
\section{Conclusion of Proof}
Let $\{G_n\}$ be any sequence of minimizing subgraphs with $|G_n| = n$. As per the proof outline in the introduction, all that we need to complete the proof of Theorem \ref{mainTheo} are the following two lemmas:
\begin{lemm}The symmetric difference of the $B^{\ell_1}_{2/\sqrt{n}}\left(\mathbf{G_n}^*\right)$ and $\mathbf{G_n}^*$ converges to $0$ as $n \to\infty$.
\end{lemm}
\begin{proof}Choose $C > 0$ such that $\mathbf{G_n}^* $ is contained in a square of side length $C$ for all $n$. Since $\mathbf{G_n}^*$ is a union of squares and all horizontal and vertical paths are continuous, we conclude that the perimeter of $\mathbf{G_n}^*$ is less than $4C$ for all $n$. The edge of a square in $\mathbf{G_n}^*$ has length $1/\sqrt{n}$. For an edge $e$ on the boundary of $\mathbf{G_n}^*$, let $B_e = \{x \in \RR^2\nin \mathbf{G_n}: \vv x - e_0\vv_{\ell_1} \leq 2/\sqrt{n}\text{ for some }e_0 \in e\}$. Simple geometric considerations imply that $|B_e| \leq M/n$ for some fixed constant $M > 0$. Since the perimeter of $\mathbf{G_n}^*$ is less than $4C$, there are at most $4C\sqrt{n}$ edges on the boundary of $\mathbf{G_n}^*$. Since $B^{\ell_1}_{2/\sqrt{n}}\left(\mathbf{G_n}^*\right) \nin \mathbf{G_n}^* = \cup_{e \in \pa \mathbf{G_n}^*} B_e$, we conclude that
\[|B^{\ell_1}_{2/\sqrt{n}}\left(\mathbf{G_n}^*\right) \nin \mathbf{G_n}^*| \leq \frac{4CM}{\sqrt{n}} \to 0 \text{ as } n\to\infty\]
\end{proof}
\begin{lemm}The sequence $\left\{\lam\left(B^{\ell_1}_{2/\sqrt{n}}\left(\mathbf{G_n}^*\right)\right)\right\}$ is uniformly bounded.
\end{lemm}
\begin{proof} Since $\mathbf{G_n}^* \subset B^{\ell_1}_{2/\sqrt{n}}\left(\mathbf{G_n}^*\right)$ it suffices to show that $\lam\left(\mathbf{G_n}^*\right)$ is uniformly bounded. For the sake of contradiction, assume that we have a subsequence $\{G_{n_k}\}$ with $\lim_{k\to\infty}\lam\left(\mathbf{G_{n_k}}^*\right) = \infty$. Recall that the the first eigenvalue of a square with side length $\delta$ is $2\pi^2/\delta^2$. Hence, for any $\delta$ we can find $K$ such that $k \geq K$ implies that $\mathbf{G_{n_k}}^*$ contains no square with side length $\delta$. Previously we showed that there always exist horizontal and vertical paths in $G_{n_k}$ that ``wall-in'' $G_{n_k}$. Thus we can find a horizontal segment $s_1$ and a vertical segment $s_2$ in $\mathbf{G_{n_k}}^*$ such that $\mathbf{G_{n_k}}^*$ lies in the rectangle determined by $s_1$ and $s_2$. All horizontal and vertical paths in $\mathbf{G_{n_k}}^*$ must be continuous. Thus, if no square with side length $\delta$ lies in $\mathbf{G_{n_k}}^*$, every point in $\mathbf{G_{n_k}}^*$ must have $\ell_1$ distance less than $\delta$ to either $s_1$ or $s_2$. In turn this implies that $1 = |\mathbf{G_{n_k}}^*| \leq 2s_1\delta+2s_2\delta$. This gives $\max(s_1,s_2) \geq 1/4\delta$. But, $\max(s_1,s_2)$ is less than the diameter of $\mathbf{G_{n_k}}^*$, which is uniformly bounded. Taking $\delta \to 0$ gives a contradiction.
\end{proof}
\section{Final Remarks}
In \cite{n15}, versions of Melas' stability theorem are generalized to domains in $\RR^d$. Hence, the main difficulties in generalizing Theorem \ref{mainTheo} to $\ZZ^d$ most likely consist of notational headaches and integrating on $d$ dimensional simplicies (Section \ref{analysis}).

However, when this project was started, the original goal of the author was the following:
\begin{conj}\label{mainConj}
Let $\{G_n\}$ be any sequence of subgraphs in $\ZZ^d$ such that $|G_n| = n$ and $\lam_D(G_n) = \lam_D^{(n)}$. Let $D \subset \RR^d$ denote the unit ball. Then, after possibly translating the $G_n$, the Hausdorff distance of $\mathbf{G_n}^*$ and $D$ converges to $0$ as $n\to\infty$.
\end{conj}
Hausdorff convergence is of course a much stronger requirement than asking that the symmetric difference has measure converging to $0$. Morally, the main problem is that Theorem \ref{mainTheo} (and all results in this paper concerning the geometry of minimizing subgraphs) do not preclude the possibility that minimizing subgraphs look like ``balls with long thin tails,'' i.e. balls with very thin tubes coming out. Theorem \ref{mainTheo} and the diameter bounds we have established only force these tubes to become thinner and thinner; so that, ``in the limit'' the $\mathbf{G_n}^*$'s become a ball with some line segments attached. Of course Conjecture \ref{mainConj} would preclude such tails.

A direct attack on Conjecture \ref{mainConj} along the lines of this paper will not work, because the Faber-Krahn inequality is \emph{not} stable with respect to Hausdorff distance unless we restrict ourselves to convex domains (balls with long thin tails again). One possibility is to prove that for every sequence of minimizing subgraphs $\{G_{n_k}\}$, $\{\mathbf{G_{n_k}}^*\}$ contains a Hausdorff convergent subsequence. Then some applications of Theorem \ref{strongFaberKrahn} would prove Conjecture \ref{mainConj}. However, such a compactness result has proved elusive. Another route is to leverage Theorem \ref{mainTheo} to analyze principle eigenfunctions of minimizing subgraphs and proceed from there. For example, with $n$ large, a minimizing subgraph $G_n$ has a ball-part and a tube-part which ``sticks out.'' Then one might try to show that on one hand, by approximation with eigenfunctions of the regular Laplacian, the values of the eigenfunction on the boundary of the ball-part are much larger than the values of the eigenfunction on the boundary of the tube-part. On the other hand, the Rayleigh quotient can be used to establish bounds on how much the eigenfunction can vary across the boundary of a minimizing subgraph.
\section{Acknowledgements}
I would like to thank Professor Mazzeo very much for guiding me through the honors thesis process. He ability to quickly decide (correctly) if an idea was worth pursing was extremely useful! He also displayed great amounts patience while I explained many half-baked ideas.
\section{Appendix I: Discrete Approximations of the Unit Disk}
We will construct subgraphs $\{D_n\}$ such that $|D_n| = n$ and $\lim_{n\to\infty}\lam(\mathbf{D_n}^*) \to \lam(D)$, where $D$ is the disk of area $1$. These $D_n$ can be thought of as discrete approximations to $D$. They are constructed inductively. Let $D_1 = \{(0,0)\}$. Now assume that $D_{n-1}$ has been constructed. Choose any $x$ such that
\[|x| = \min_{y \in \ZZ^2 - D_{n-1}}|y|\]
Then let $D_n$ be the subgraph with vertex set $D_{n-1} \cup \{x\}$.
It is clear that $|D_n| = n$. Hence, it only remains to show
\begin{prop}\label{discreteDisk}$\lam(\mathbf{D_n}^*) \to \lam(D)$ as $n\to\infty$.
\end{prop}
The proof of this relies on the following lemma due to Gauss.
\begin{lemm}\label{gLemm}Let $B_r \subset \RR^2$ be the disk of area $r$ centered at the origin. For any $n \in \ZZ$ set $O_n = B_n \cap \ZZ^2$. Then $|O_n| = n + O(\sqrt{n})$
\end{lemm}
\begin{proof}For any $x \in \ZZ^2$, let $S_x$ denoted the square of area $1$ centered at $x$. The radius of $B_n$ is $\sqrt{n/\pi}$. For any $x \in B_n$, $S_x$ is contained in the disk of radius $\sqrt{n/\pi} + 1/2$. Hence $\mathbf{O_n}$ is contained in a disk of radius $\sqrt{n/\pi} + 1/2$. By construction, the area of $\mathbf{O_n}$ is equal to $|O_n|$. Thus we have
\[|O_n| \leq \pi(\sqrt{n/\pi} + 1/2)^2 = n + \sqrt{n\pi} + \pi/4\]
We can apply the same idea in reverse to get an inequality in the opposite direction. That is, let $y \in B_n$ and suppose that $y \not\in C_x$ for any $x \in O_n$. Since the squares $C_x$ for $x \in \ZZ^2$ completely tile $\RR^2$ we conclude that there exists $x \in \ZZ^2-B_n$ with $y \in S_x$. This implies that $y$ does not lie in the disk of radius $\sqrt{n/\pi}-1/2$ centered at the origin. Hence the disk of radius $\sqrt{n/\pi} - 1/2$ is completely covered by $\mathbf{O_n}$. This gives
\[|O_n| \geq \pi(\sqrt{n/\pi} - 1/2)^2 = n - \sqrt{n\pi} + \pi/4\]
\end{proof}
Recall that for any constant $\alpha > 0$ and a bounded open domain $\Omega$, $\lam(\alpha \Omega) = \lam(\Omega)/\alpha^2$. Hence, to prove Proposition \ref{discreteDisk}, it suffices to prove that $\frac{\lam(D)}{n+o(n)} \leq \lam(\mathbf{D_n}) \leq  \frac{\lam(D)}{n - o(n)}$.
\begin{proof}Via Lemma \ref{gLemm} choose $C > 0$ so that
\[\Big||B_n \cap \ZZ^2| - n\Big| \leq C\sqrt{n}\]
Then, for any $m$ such that $m - C\sqrt{m} \geq n$, we will have $D_n \subset B_m$. Solving the relevant quadratic equation reveals that we can take
\[m = \frac{4n + 2C^2 + 2C\sqrt{C^2+4n}}{4} = n + O(\sqrt{n})\]
The radius of $B_m$ is then
\[\sqrt{\frac{m}{\pi}} = \sqrt{\frac{n}{\pi}} + o(\sqrt{n})\]
Now let
\[r_1 = \sqrt{\frac{m}{\pi}} + 1/2 = \sqrt{\frac{n}{\pi}} + o(\sqrt{n})\]
$\mathbf{D_n}$ is contained in the disk of radius $r_1$. A completely analogous argument produces
\[r_2 = \sqrt{\frac{n}{\pi}} + o(\sqrt{n})\]
such that the ball of radius $r_2$ is contained in $\mathbf{D_n}$. Since $\Omega_1 \subset \Omega_2$ implies $\lam(\Omega_1) \geq \lam(\Omega_2)$, we get
\[\frac{\lam(D)}{n+o(n)} \leq \lam(\mathbf{D_n}) \leq  \frac{\lam(D)}{n - o(n)}\]
\end{proof}
\section{Appendix II: Proof of Lemma \ref{annoyingLemm}}
Our goal is to prove
\[\int_{\RR^2}|\nabla u|^2\ dxdy - \sum_{i \sim_{h\ZZ^2} j}(v(i)-v(j))^2 = \]
\[\frac{1}{h}\sum_{g\in h\ZZ^2}\int_{-\frac{h}{2}}^{\frac{3h}{2}}\int_{-\frac{h}{2}}^{\frac{h}{2}}\psi(x)\left(\frac{\pa u_g}{\pa x}(x,y) - \frac{1}{h}\left(v_g(h,0) - v_g(0,0)\right)\right)^2\ dxdy + \]
\[\frac{1}{h}\sum_{g\in h\ZZ^2}\int_{-\frac{h}{2}}^{\frac{3h}{2}}\int_{-\frac{h}{2}}^{\frac{h}{2}}\psi(y)\left(\frac{\pa u_g}{\pa y}(x,y) - \frac{1}{h}\left(v_g(0,h) - v_g(0,0)\right)\right)^2\ dydx\]
\begin{proof}
For any function $f$
\begin{align*}
\int_{-h/2}^{3h/2}\int_{-h/2}^{h/2}\psi(x)f(x,y)\ dxdy &= \\
\int_{-h/2}^{h/2}\left[\int_{-h/2}^{h/2}\left(x+\frac{h}{2}\right)f(x,y)\ dx + \int_{h/2}^{3h/2}\left(\frac{3h}{2}-x\right)f(x,y)\ dx\right]\ dy &=\\
\int_{-h/2}^{h/2}\int_{-h/2}^{h/2}\left(x+\frac{h}{2}\right)f(x,y)\ dxdy + \int_{-h/2}^{h/2}\int_{-h/2}^{h/2}\left(\frac{h}{2}-x\right)f(x-h,y)\ dxdy&
\end{align*}
Next
\begin{align*}
\frac{1}{h}\sum_{g\in h\ZZ^2}\int_{-h/2}^{3h/2}\int_{-h/2}^{h/2}\psi(x)\frac{\pa u_g}{\pa x}^2(x,y)\ dxdy &=\\
\frac{1}{h}\sum_{g \in h\ZZ^2}\Bigg[\int_{-h/2}^{h/2}\int_{-h/2}^{h/2}\left(x + \frac{h}{2}\right)\frac{\pa u_g}{\pa x}^2(x,y)\ dxdy &\\
+ \int_{-h/2}^{h/2}\int_{-h/2}^{h/2}\left(\frac{h}{2}-x\right)\frac{\pa u_g}{\pa x}^2(x-h,y)\ dxdy\Bigg]&=\\
\frac{1}{h}\sum_{i,j = -\infty}^{\infty}\Bigg[\int_{-h/2}^{h/2}\int_{-h/2}^{h/2}\left(x + \frac{h}{2}\right)\frac{\pa u}{\pa x}^2(ih+x,jh+y)\ dxdy &\\
+ \int_{-h/2}^{h/2}\int_{-h/2}^{h/2}\left(\frac{h}{2}-x\right)\frac{\pa u}{\pa x}^2((i-1)h+x,jh+y)\ dxdy\Bigg]&=\\
\frac{1}{h}\Bigg[\sum_{i,j = -\infty}^{\infty}\int_{-h/2}^{h/2}\int_{-h/2}^{h/2}\left(x + \frac{h}{2}\right)\frac{\pa u}{\pa x}^2(ih+x,jh+y)\ dxdy &\\
+ \sum_{i,j=-\infty}^{\infty}\int_{-h/2}^{h/2}\int_{-h/2}^{h/2}\left(\frac{h}{2}-x\right)\frac{\pa u}{\pa x}^2(ih+x,jh+y)\ dxdy\Bigg]&=\\
\frac{1}{h}\sum_{i,j=-\infty}^{\infty}\int_{-h/2}^{h/2}\int_{-h/2}^{h/2}h\frac{\pa u}{\pa x}^2(ih+x,jh+y)\ dxdy &=\\
\int_{\RR^2}\frac{\pa u}{\pa x}^2\ dxdy
\end{align*}
Similarly we have
\[\frac{1}{h}\sum_{g\in h\ZZ^2}\int_{-h/2}^{3h/2}\int_{-h/2}^{h/2}\psi(y)\frac{\pa u_g}{\pa y}^2(x,y)\ dydx = \int_{\RR^2}\frac{\pa u}{\pa y}^2\ dxdy\]
and
\[\frac{1}{h}\sum_{g \in h\ZZ^2}\int_{-h/2}^{3h/2}\int_{-h/2}^{h/2}\psi(x)\frac{(v_g(h,0)-v_g(0,0))^2}{h^2}\ dxdy + \]
\[\frac{1}{h}\sum_{g \in h\ZZ^2}\int_{-h/2}^{3h/2}\int_{-h/2}^{h/2}\psi(y)\frac{(v_g(0,h)-v_g(0,0))^2}{h^2}\ dydx = \]
\[\sum_{i\sim_{h\ZZ^2} j}(v(i)-v(j))^2\]
Integration by parts gives the following identity
\begin{align*}
v_g(h,0) - v_g(0,0) &= h^2\left(\int_{h/2}^{3h/2}\int_{-h/2}^{h/2}u_g\ dxdy - \int_{-h/2}^{h/2}\int_{-h/2}^{h/2}u_g\ dxdy\right) \\
&= h^2\int_{-h/2}^{h/2}\left(\int_{h/2}^{3h/2}u_g\ dx - \int_{-h/2}^{h/2}u_g\ dx\right)\ dy\\
&= -h^2\int_{-h/2}^{h/2}\left(\int_{h/2}^{3h/2}\psi'(x)u_g\ dx  - \int_{-h/2}^{h/2}\psi'(x)u_g\ dx\right)\ dy \\
&= h^2\int_{-h/2}^{h/2}\int_{-h/2}^{3h/2}\Bigg(\psi(x)\frac{\pa u_g}{\pa x}\ dx \\
& \ -\psi(3h/2)u_g(3h/2,y) + \psi(h/2)u_g(h/2,y) \\
& \ - \psi(h/2)u_g(h/2,y) + \psi(-h/2)u_g(-h/2,y)\ dy\Bigg) \\
&= h^2\int_{-h/2}^{3h/2}\int_{-h/2}^{h/2}\psi(x)\frac{\pa u_g}{\pa x}\ dxdy
\end{align*}
Thus
\begin{align*}
\frac{-2}{h^2}\sum_{g\in h\ZZ^2}\Bigg[\int_{-h/2}^{3h/2}\int_{-h/2}^{h/2}\psi(x)\frac{\pa u_g}{\pa x}(v_g(h,0)-v_g(0,0))\ dxdy + &\\
\int_{-h/2}^{3h/2}\int_{-h/2}^{h/2}\psi(y)\frac{\pa u_g}{\pa y}(v_g(0,h)-v_g(0,0))\ dydx \Bigg]&=\\
-2\sum_{i\sim_{h\ZZ^2} j}(v(i)-v(j))^2&
\end{align*}
Now we are ready to put everything together:
\begin{align*}
\frac{1}{h}\sum_{g\in h\ZZ^2}\int_{-\frac{h}{2}}^{\frac{3h}{2}}\int_{-\frac{h}{2}}^{\frac{h}{2}}\psi(x)\left(\frac{\pa u_g}{\pa x}(x,y) - \frac{1}{h}\left(v_g(h,0) - v_g(0,0)\right)\right)^2\ dxdy + &\\
\frac{1}{h}\sum_{g\in h\ZZ^2}\int_{-\frac{h}{2}}^{\frac{3h}{2}}\int_{-\frac{h}{2}}^{\frac{h}{2}}\psi(y)\left(\frac{\pa u_g}{\pa y}(x,y) - \frac{1}{h}\left(v_g(0,h) - v_g(0,0)\right)\right)^2\ dydx&=\\
\frac{1}{h}\sum_{g\in h\ZZ^2}\int_{-h/2}^{3h/2}\int_{-h/2}^{h/2}\psi(x)\frac{\pa u_g}{\pa x}^2(x,y)\ dxdy + &\\
\frac{1}{h}\sum_{g\in h\ZZ^2}\int_{-h/2}^{3h/2}\int_{-h/2}^{h/2}\psi(y)\frac{\pa u_g}{\pa y}^2(x,y)\ dydx + &\\
\frac{1}{h}\sum_{g \in h\ZZ^2}\int_{-h/2}^{3h/2}\int_{-h/2}^{h/2}\psi(x)\frac{(v_g(h,0)-v_g(0,0))^2}{h^2}\ dxdy + &\\
\frac{1}{h}\sum_{g \in h\ZZ^2}\int_{-h/2}^{3h/2}\int_{-h/2}^{h/2}\psi(y)\frac{(v_g(0,h)-v_g(0,0))^2}{h^2}\ dydx + &\\
\frac{-2}{h^2}\sum_{g\in h\ZZ^2}\Bigg[\int_{-h/2}^{3h/2}\int_{-h/2}^{h/2}\psi(x)\frac{\pa u_g}{\pa x}(v_g(h,0)-v_g(0,0))\ dxdy + &\\
\int_{-h/2}^{3h/2}\int_{-h/2}^{h/2}\psi(y)\frac{\pa u_g}{\pa y}(v_g(0,h)-v_g(0,0))\ dydx \Bigg] &=\\
\int_{\RR^2}\frac{\pa u}{\pa x}^2\ dxdy + \int_{\RR^2}\frac{\pa u}{\pa y}^2\ dxdy + &\\
 + \sum_{i\sim_{h\ZZ^2} j}(v(i)-v(j))^2 -2\sum_{i\sim_{h\ZZ^2} j}(v(i)-v(j))^2 &= \\
\int_{\RR^2}|\nabla u|^2\ dxdy - \sum_{i \sim_{h\ZZ^2} j}(v(i)-v(j))^2
\end{align*}
\end{proof}
\pagebreak

\end{document}